\documentclass[11pt,twoside,a4paper]{article}
\usepackage{color}
\newcommand{\fleche}[1]{\stackrel{#1}{\longrightarrow}}
\usepackage{feynmp}	
\usepackage{amsmath, amssymb}
\usepackage[latin1]{inputenc}
\usepackage[T1]{fontenc}   
\usepackage[english]{babel}
\newcommand{\tdun}[1]{\begin{picture}(10,5)(-2,-1)
\put(0,0){\circle*{2}}
\put(3,-2){\tiny #1}
\end{picture}}

\newcommand{\tddeux}[2]{\begin{picture}(12,5)(0,-1)
\put(3,0){\circle*{2}}
\put(3,5){\circle*{2}}
\put(3,0){\line(0,1){5}}
\put(6,-2){\tiny #1}
\put(6,3){\tiny #2}
\end{picture}}

\newcommand{\tdtroisun}[3]{\begin{picture}(20,12)(-5,-1)
\put(3,0){\circle*{2}}
\put(6,7){\circle*{2}}
\put(0,7){\circle*{2}}
\put(-0.65,0){$\vee$}
\put(5,-2){\tiny #1}
\put(9,5){\tiny #2}
\put(-5,5){\tiny #3}
\end{picture}}

\newcommand{\tdtroisdeux}[3]{\begin{picture}(12,15)(-2,-1)
\put(0,0){\circle*{2}}
\put(0,5){\circle*{2}}
\put(0,10){\circle*{2}}
\put(0,0){\line(0,1){5}}
\put(0,5){\line(0,1){5}}
\put(3,-2){\tiny #1}
\put(3,3){\tiny #2}
\put(3,9){\tiny #3}
\end{picture}}

\newcommand{\tdquatreun}[4]{\begin{picture}(20,12)(-5,-1)
\put(3,0){\circle*{2}}
\put(6,7){\circle*{2}}
\put(0,7){\circle*{2}}
\put(3,7){\circle*{2}}
\put(-0.6,0){$\vee$}
\put(3,0){\line(0,1){7}}
\put(5,-2){\tiny #1}
\put(8.5,5){\tiny #2}
\put(1,10){\tiny #3}
\put(-5,5){\tiny #4}
\end{picture}}

\newcommand{\tdquatredeux}[4]{\begin{picture}(20,20)(-5,-1)
\put(3,0){\circle*{2}}
\put(6,7){\circle*{2}}
\put(0,7){\circle*{2}}
\put(0,14){\circle*{2}}
\put(-.65,0){$\vee$}
\put(0,7){\line(0,1){7}}
\put(5,-2){\tiny #1}
\put(9,5){\tiny #2}
\put(-5,5){\tiny #3}
\put(-5,12){\tiny #4}
\end{picture}}

\newcommand{\tdquatretrois}[4]{\begin{picture}(20,20)(-5,-1)
\put(3,0){\circle*{2}}
\put(6,7){\circle*{2}}
\put(0,7){\circle*{2}}
\put(6,14){\circle*{2}}
\put(-.65,0){$\vee$}
\put(6,7){\line(0,1){7}}
\put(5,-2){\tiny #1}
\put(9,5){\tiny #2}
\put(-5,5){\tiny #4}
\put(9,12){\tiny #3}
\end{picture}}

\newcommand{\tdquatrequatre}[4]{\begin{picture}(20,14)(-5,-1)
\put(3,5){\circle*{2}}
\put(6,12){\circle*{2}}
\put(0,12){\circle*{2}}
\put(3,0){\circle*{2}}
\put(-.65,5){$\vee$}
\put(3,0){\line(0,1){5}}
\put(6,-3){\tiny #1}
\put(6,4){\tiny #2}
\put(9,12){\tiny #3}
\put(-5,12){\tiny #4}
\end{picture}}

\newcommand{\tdquatrecinq}[4]{\begin{picture}(12,19)(-2,-1)
\put(0,0){\circle*{2}}
\put(0,5){\circle*{2}}
\put(0,10){\circle*{2}}
\put(0,15){\circle*{2}}
\put(0,0){\line(0,1){5}}
\put(0,5){\line(0,1){5}}
\put(0,10){\line(0,1){5}}
\put(3,-2){\tiny #1}
\put(3,3){\tiny #2}
\put(3,9){\tiny #3}
\put(3,14){\tiny #4}
\end{picture}}

\newcommand{\edgeelectronun}{\begin{fmffile}{F1}\parbox{12mm}{\begin{fmfgraph}(30,20)
\fmfleft{i}\fmfright{o}\fmf{electron}{i,o}
\end{fmfgraph}}\end{fmffile}}

\newcommand{\edgeelectrondeux}{\begin{fmffile}{F2}\parbox{12mm}{\begin{fmfgraph}(30,20)
\fmfleft{i}\fmfright{o}\fmf{electron}{o,i}
\end{fmfgraph}}\end{fmffile}}

\newcommand{\edgephoton}{\begin{fmffile}{F3}\parbox{12mm}{\begin{fmfgraph}(30,20)
\fmfleft{i}\fmfright{o}\fmf{photon}{i,o}
\end{fmfgraph}}\end{fmffile}}

\newcommand{\edgegluon}{\begin{fmffile}{F4}\parbox{12mm}{\begin{fmfgraph}(30,20)
\fmfleft{i}\fmfright{o}\fmf{gluon}{i,o}
\end{fmfgraph}}\end{fmffile}}

\newcommand{\edgeghostun}{\begin{fmffile}{F5}\parbox{12mm}{\begin{fmfgraph}(30,20)
\fmfleft{i}\fmfright{o}\fmf{ghost}{i,o}
\end{fmfgraph}}\end{fmffile}}

\newcommand{\edgeghostdeux}{\begin{fmffile}{F6}\parbox{12mm}{\begin{fmfgraph}(30,20)
\fmfleft{i}\fmfright{o}\fmf{ghost}{o,i}
\end{fmfgraph}}\end{fmffile}}

\newcommand{\edgephi}{\begin{fmffile}{F12}\parbox{12mm}{\begin{fmfgraph}(30,20)
\fmfleft{i}\fmfright{o}\fmf{plain}{o,i}
\end{fmfgraph}}\end{fmffile}}

\newcommand{\vertexQED}{\begin{fmffile}{F7}\parbox{15mm}{\begin{fmfgraph}(40,30)
\fmfleft{i1}\fmfright{o1,o2}\fmf{photon}{i1,v}\fmf{electron}{o1,v,o2}
\end{fmfgraph}}\end{fmffile}}

\newcommand{\vertexQCDun}{\begin{fmffile}{F8}\parbox{15mm}{\begin{fmfgraph}(40,30)
\fmfleft{i1}\fmfright{o1,o2}\fmf{gluon}{i1,v}\fmf{electron}{o1,v,o2}
\end{fmfgraph}}\end{fmffile}}

\newcommand{\vertexQCDdeux}{\begin{fmffile}{F9}\parbox{15mm}{\begin{fmfgraph}(40,30)
\fmfleft{i1}\fmfright{o1,o2}\fmf{gluon}{i1,v}\fmf{ghost}{o1,v,o2}
\end{fmfgraph}}\end{fmffile}}

\newcommand{\vertexQCDtrois}{\begin{fmffile}{F10}\parbox{15mm}{\begin{fmfgraph}(40,30)
\fmfleft{i1}\fmfright{o1,o2}\fmf{gluon}{i1,v}\fmf{gluon}{o1,v,o2}
\end{fmfgraph}}\end{fmffile}}

\newcommand{\vertexQCDquatre}{\begin{fmffile}{F11}\parbox{15mm}{\begin{fmfgraph}(40,30)
\fmfleft{i1,i2}\fmfright{o1,o2}\fmf{gluon}{i2,v,i1}\fmf{gluon}{o1,v,o2}
\end{fmfgraph}}\end{fmffile}}

\newcommand{\vertexphi}{\begin{fmffile}{F60}\parbox{15mm}{\begin{fmfgraph}(40,30)
\fmfleft{i1}\fmfright{o1,o2}\fmf{plain}{i1,v}\fmf{plain}{o1,v,o2}
\end{fmfgraph}}\end{fmffile}}

\newcommand{\QEDun}{\begin{fmffile}{F13}\parbox{11mm}{\begin{fmfgraph}(40,40)
\fmfleft{i}\fmfright{o1,o2}\fmf{fermion}{o1,v1,v2,v3,o2}\fmf{photon}{i,v2}
\fmffreeze \fmf{photon}{v1,v3}\end{fmfgraph}}\end{fmffile}}

\newcommand{\QEDdeux}{\begin{fmffile}{F14}\parbox{15mm}{\begin{fmfgraph}(40,40)
\fmfleft{i}\fmfright{o}\fmf{photon}{i,v1}\fmf{photon}{v2,o}\fmf{fermion,left,tension=.2}{v1,v2,v1}
\end{fmfgraph}}\end{fmffile}}

\newcommand{\QEDtrois}{\begin{fmffile}{F15}\parbox{15mm}{\begin{fmfgraph}(40,30)
\fmfleft{i}\fmfright{o}\fmf{fermion}{i,v1,v2,o}\fmf{photon,left,tension=0. }{v1,v2}
\end{fmfgraph}}\end{fmffile}}

\newcommand{\QEDquatre}{\begin{fmffile}{F16}\parbox{18mm}{\begin{fmfgraph}(50,30)
\fmfleft{i}\fmfright{o}\fmf{photon}{i,v1}\fmf{photon}{v3,o}\fmf{fermion,tension=.5}{v1,v2,v3,v4,v1}
\fmf{photon}{v2,v4}\fmffixed{(0,h)}{v2,v4}
\end{fmfgraph}}\end{fmffile}}

\newcommand{\QEDcinq}{\begin{fmffile}{F17}\parbox{12mm}{\begin{fmfgraph}(40,40)
\fmfleft{i}\fmfright{o1,o2}\fmf{fermion}{o1,v1,v2,v3,v4,v5,o2}\fmf{photon}{i,v3}
\fmffreeze\fmf{photon}{v1,v5}\fmffreeze\fmf{photon}{v2,v4}
\end{fmfgraph}}\end{fmffile}}

\newcommand{\QEDsept}{\begin{fmffile}{F19}\parbox{12mm}{\begin{fmfgraph}(40,40)
\fmfleft{i}\fmfright{o1,o2}\fmf{fermion}{o1,v1,v2,v3,v4,v5,o2}\fmf{photon}{i,v4}
\fmffreeze\fmf{photon,left,tension=0.2}{v1,v3}\fmffreeze\fmf{photon}{v2,v5}
\end{fmfgraph}}\end{fmffile}}

\newcommand{\QEDhuit}{\begin{fmffile}{F20}\parbox{14mm}{\begin{fmfgraph}(50,50)
\fmfleft{i}\fmfright{o1,o2}\fmf{fermion}{o1,v1,v2,v3,v4,v5,o2}\fmf{photon}{i,v4}
\fmffreeze\fmf{photon,left,tension=0.3}{v2,v3}\fmffreeze\fmf{photon}{v1,v5}
\end{fmfgraph}}\end{fmffile}}

\newcommand{\QEDdix}{\begin{fmffile}{F22}\parbox{14mm}{\begin{fmfgraph}(50,50)
\fmfleft{i}\fmfright{o1,o2}\fmf{fermion}{o1,v1,v2,v3,v4,v5,o2}\fmf{photon}{i,v3}
\fmffreeze\fmf{photon}{v1,v4}\fmffreeze\fmf{photon}{v2,v5}
\end{fmfgraph}}\end{fmffile}}

\newcommand{\QEDonze}{\begin{fmffile}{F23}\parbox{18mm}{\begin{fmfgraph}(50,50)
\fmfleft{i}\fmfright{o1,o2}\fmf{fermion}{o1,v1}\fmf{plain}{v1,w1} \fmf{fermion}{w1,v2} \fmf{plain}{v2,w2} \fmf{fermion}{w2,v3}
\fmf{fermion}{v3,o2}\fmf{photon}{i,v2}\fmffreeze\fmf{photon}{v1,v4}\fmf{photon}{v3,v5}\fmf{fermion,left,tension=0.2}{v4,v5,v4}
\end{fmfgraph}}\end{fmffile}}

\newcommand{\QEDdouze}{\begin{fmffile}{F24}\parbox{18mm}{\begin{fmfgraph}(50,60)
\fmfleft{i}\fmfright{o}\fmf{fermion}{i,v1,v2,o}\fmf{photon,left}{v1,w1} \fmf{photon,right}{v2,w2}
\fmf{fermion,left,tension=.2}{w1,w2,w1}\fmffixed{(0,.2h)}{v1,w1}\fmffixed{(0,.2h)}{v2,w2}
\end{fmfgraph}}\end{fmffile}}

\newcommand{\QEDtreize}{\begin{fmffile}{F25}\parbox{18mm}{\begin{fmfgraph}(50,50)
\fmfleft{i}\fmfright{o}\fmf{fermion}{i,v1,v2,v3,v4,o} \fmf{photon,left}{v1,v3}\fmf{photon,right}{v2,v4}
\fmffixed{(.2h,0)}{v1,v2}\fmffixed{(.2h,0)}{v2,v3}\fmffixed{(.2h,0)}{v3,v4}
\end{fmfgraph}}\end{fmffile}}

\newcommand{\QEDquatorze}{\begin{fmffile}{F26}\parbox{18mm}{\begin{fmfgraph}(50,50)
\fmfleft{i}\fmfright{o}\fmf{photon}{i,v1}\fmf{photon}{v4,o}\fmf{fermion}{v1,v2,v3,v4,v1}\fmf{photon,left}{v2,v3}
\fmffixed{(0.3h,0)}{v2,v3}\fmffixed{(0,.3h)}{v1,v2}\fmffixed{(0,.3h)}{v4,v3}
\end{fmfgraph}}\end{fmffile}}

\newcommand{\QCDun}{\begin{fmffile}{F27}\parbox{15mm}{\begin{fmfgraph}(50,50)
\fmfleft{i}\fmfright{o1,o2}\fmf{fermion}{o1,v1,v2,v3,o2}\fmf{gluon}{i,v2}
\fmffreeze \fmf{gluon}{v1,v3}\end{fmfgraph}}\end{fmffile}}

\newcommand{\QCDunbis}{\begin{fmffile}{F39}\parbox{15mm}{\begin{fmfgraph}(50,50)
\fmfleft{i}\fmfright{o1,o2}\fmf{ghost}{o1,v1,v2,v3,o2}\fmf{gluon}{i,v2}
\fmffreeze \fmf{gluon}{v1,v3}\end{fmfgraph}}\end{fmffile}}

\newcommand{\QCDunter}{\begin{fmffile}{F41}\parbox{25mm}{\begin{fmfgraph}(50,50)
\fmfleft{i}\fmfright{o1,o2}\fmf{gluon}{o2,v3,v2,v1,o1}\fmf{gluon}{i,v2}
\fmffreeze \fmf{gluon}{v1,v3}\end{fmfgraph}}\end{fmffile}}

\newcommand{\QCDdeux}{\begin{fmffile}{F28}\parbox{15mm}{\begin{fmfgraph}(40,40)
\fmfleft{i}\fmfright{o}\fmf{gluon}{i,v1}\fmf{gluon}{v2,o}\fmf{fermion,left,tension=.2}{v1,v2,v1}
\end{fmfgraph}}\end{fmffile}}

\newcommand{\QCDdeuxbis}{\begin{fmffile}{F42}\parbox{15mm}{\begin{fmfgraph}(40,40)
\fmfleft{i}\fmfright{o}\fmf{gluon}{i,v1}\fmf{gluon}{v2,o}\fmf{ghost,left,tension=.2}{v1,v2,v1}
\end{fmfgraph}}\end{fmffile}}

\newcommand{\QCDtrois}{\begin{fmffile}{F29}\parbox{20mm}{\begin{fmfgraph}(50,40)
\fmfleft{i}\fmfright{o}\fmf{fermion}{i,v1,v2,o}\fmf{gluon,left}{v1,v2}\fmffixed{(.6h,0)}{v1,v2}
\end{fmfgraph}}\end{fmffile}}

\newcommand{\QCDtroisbis}{\begin{fmffile}{F43}\parbox{20mm}{\begin{fmfgraph}(50,40)
\fmfleft{i}\fmfright{o}\fmf{ghost}{i,v1,v2,o}\fmf{gluon,left}{v1,v2}\fmffixed{(.6h,0)}{v1,v2}
\end{fmfgraph}}\end{fmffile}}

\newcommand{\QCDquatre}{\begin{fmffile}{F30}\parbox{23mm}{\begin{fmfgraph}(60,40)
\fmfleft{i}\fmfright{o}\fmf{gluon}{i,v1}\fmf{gluon}{v3,o}\fmf{fermion,tension=0.5}{v1,v2,v3,v4,v1}
\fmf{gluon}{v2,v4}\fmffixed{(0,h)}{v2,v4}
\end{fmfgraph}}\end{fmffile}}

\newcommand{\QCDcinq}{\begin{fmffile}{F31}\parbox{18mm}{\begin{fmfgraph}(50,50)
\fmfleft{i}\fmfright{o1,o2}\fmf{ghost}{o1,v1,v2,v3,v4,v5,o2}\fmf{gluon}{i,v3}
\fmffreeze\fmf{gluon}{v1,v5}\fmffreeze\fmf{gluon}{v2,v4}
\end{fmfgraph}}\end{fmffile}}

\newcommand{\QCDsept}{\begin{fmffile}{F33}\parbox{15mm}{\begin{fmfgraph}(50,50)
\fmfleft{i}\fmfright{o1,o2}\fmf{fermion}{o1,v1,v2,v3,v4,v5,o2}\fmf{gluon}{i,v4}
\fmffreeze\fmf{gluon,left,tension=0.2}{v1,v3}\fmffreeze\fmf{gluon}{v2,v5}
\end{fmfgraph}}\end{fmffile}}

\newcommand{\QCDhuit}{\begin{fmffile}{F34}\parbox{20mm}{\begin{fmfgraph}(60,60)
\fmfleft{i}\fmfright{o1,o2}\fmf{ghost}{o1,v1,v2,v3,v4,v5,o2}\fmf{gluon}{i,v4}
\fmffreeze\fmf{gluon,right,tension=0.3}{v3,v2}\fmffreeze\fmf{gluon}{v1,v5}
\end{fmfgraph}}\end{fmffile}}

\newcommand{\QCDdix}{\begin{fmffile}{F36}\parbox{18mm}{\begin{fmfgraph}(60,60)
\fmfleft{i}\fmfright{o1,o2}\fmf{fermion}{o1,v1,v2,v3,v4,v5,o2}\fmf{gluon}{i,v3}
\fmffreeze\fmf{gluon}{v1,v4}\fmffreeze\fmf{gluon}{v2,v5}
\end{fmfgraph}}\end{fmffile}}

\newcommand{\QCDonze}{\begin{fmffile}{F37}\parbox{20mm}{\begin{fmfgraph}(60,60)
\fmfleft{i}\fmfright{o1,o2}\fmf{ghost}{o1,v1}\fmf{dots}{v1,w1} \fmf{ghost}{w1,v2} \fmf{dots}{v2,w2} \fmf{ghost}{w2,v3}
\fmf{ghost}{v3,o2}\fmf{gluon}{i,v2}\fmffreeze\fmf{gluon}{v1,v4}\fmf{gluon}{v3,v5}\fmf{fermion,left,tension=0.2}{v4,v5,v4}
\end{fmfgraph}}\end{fmffile}}

\newcommand{\QCDonzebis}{\begin{fmffile}{F40}\parbox{20mm}{\begin{fmfgraph}(60,60)
\fmfleft{i}\fmfright{o1,o2}\fmf{fermion}{o1,v1}\fmf{plain}{v1,w1} \fmf{fermion}{w1,v2} \fmf{plain}{v2,w2} \fmf{fermion}{w2,v3}
\fmf{fermion}{v3,o2}\fmf{gluon}{i,v2}\fmffreeze\fmf{gluon}{v1,v4}\fmf{gluon}{v3,v5}\fmf{ghost,left,tension=0.2}{v4,v5,v4}
\end{fmfgraph}}\end{fmffile}}

\newcommand{\QCDdouze}{\begin{fmffile}{F38}\parbox{18mm}{\begin{fmfgraph}(50,60)
\fmfleft{i}\fmfright{o}\fmf{fermion}{i,v1,v2,o}\fmf{gluon,left}{v1,w1} \fmf{gluon}{v2,w2}
\fmf{ghost,left,tension=.2}{w1,w2,w1}\fmffixed{(0,.3h)}{v1,w1}\fmffixed{(0,.3h)}{v2,w2}\fmffixed{(.3h,0)}{v1,v2}\fmffixed{(.3h,0)}{w1,w2}
\end{fmfgraph}}\end{fmffile}}

\newcommand{\QCDdouzebis}{\begin{fmffile}{F44}\parbox{18mm}{\begin{fmfgraph}(50,60)
\fmfleft{i}\fmfright{o}\fmf{ghost}{i,v1,v2,o}\fmf{gluon,left}{v1,w1} \fmf{gluon}{v2,w2}
\fmf{fermion,left,tension=.2}{w1,w2,w1}\fmffixed{(0,.3h)}{v1,w1}\fmffixed{(0,.3h)}{v2,w2}\fmffixed{(.3h,0)}{v1,v2}\fmffixed{(.3h,0)}{w1,w2}
\end{fmfgraph}}\end{fmffile}}

\newcommand{\QCDtreize}{\begin{fmffile}{F45}\parbox{23mm}{\begin{fmfgraph}(60,50)
\fmfleft{i}\fmfright{o}\fmf{fermion}{i,v1,v2,v3,v4,o} \fmf{gluon,left}{v1,v3}\fmf{gluon,right}{v2,v4}
\fmffixed{(.4h,0)}{v1,v2}\fmffixed{(.2h,0)}{v2,v3}\fmffixed{(.2h,0)}{v3,v4}
\end{fmfgraph}}\end{fmffile}}

\newcommand{\QCDtreizebis}{\begin{fmffile}{F46}\parbox{18mm}{\begin{fmfgraph}(50,50)
\fmfleft{i}\fmfright{o}\fmf{ghost}{i,v1,v2,v3,v4,o} \fmf{gluon,left}{v1,v3}\fmf{gluon,right}{v2,v4}
\fmffixed{(.4h,0)}{v1,v2}\fmffixed{(.2h,0)}{v2,v3}\fmffixed{(.2h,0)}{v3,v4}
\end{fmfgraph}}\end{fmffile}}

\newcommand{\QCDquatorze}{\begin{fmffile}{F47}\parbox{23mm}{\begin{fmfgraph}(60,60)
\fmfleft{i}\fmfright{o}\fmf{gluon}{i,v1}\fmf{gluon}{v4,o}\fmf{fermion}{v1,v2,v3,v4,v1}\fmf{gluon,left}{v2,v3}
\fmffixed{(0.3h,0)}{v2,v3}\fmffixed{(0,.3h)}{v1,v2}\fmffixed{(0,.3h)}{v4,v3}
\end{fmfgraph}}\end{fmffile}}

\newcommand{\QCDquinze}{\begin{fmffile}{F48}\parbox{15mm}{\begin{fmfgraph}(45,60)
\fmfleft{i1,i2}\fmfright{o1,o2}\fmf{gluon}{i1,v1,v3,o1}\fmf{gluon}{o2,v4,v2,i2}\fmf{gluon}{v2,v1}\fmf{gluon}{v3,v4}
\fmffixed{(0.3h,0)}{v1,v3}\fmffixed{(0.3h,0)}{v2,v4}
\end{fmfgraph}}\end{fmffile}}

\newcommand{\QCDseize}{\begin{fmffile}{F49}\parbox{15mm}{\begin{fmfgraph}(45,60)
\fmfleft{i1,i2}\fmfright{o1,o2}\fmf{gluon}{i1,v1,v3,o1}\fmf{gluon}{o2,v4,v2,i2}\fmf{gluon}{v2,v1}\fmf{gluon}{v3,v4}\fmf{gluon}{v1,v4}
\fmffixed{(0.3h,0)}{v1,v3}\fmffixed{(0.3h,0)}{v2,v4}
\end{fmfgraph}}\end{fmffile}}

\newcommand{\phiun}{\begin{fmffile}{F50}\parbox{11mm}{\begin{fmfgraph}(40,40)
\fmfleft{i}\fmfright{o1,o2}\fmf{plain}{o1,v1,v2,v3,o2}\fmf{plain}{i,v2}
\fmffreeze \fmf{plain}{v1,v3}\end{fmfgraph}}\end{fmffile}}

\newcommand{\phideux}{\begin{fmffile}{F51}\parbox{11mm}{\begin{fmfgraph}(30,30)
\fmfleft{i}\fmfright{o}\fmf{plain}{i,v1}\fmf{plain}{v2,o}\fmf{plain,left,tension=.2}{v1,v2,v1}
\end{fmfgraph}}\end{fmffile}}

\newcommand{\phiquatre}{\begin{fmffile}{F52}\parbox{18mm}{\begin{fmfgraph}(50,30)
\fmfleft{i}\fmfright{o}\fmf{plain}{i,v1}\fmf{plain}{v3,o}\fmf{plain,tension=0.5}{v1,v2,v3,v4,v1}
\fmf{plain}{v2,v4}\fmffixed{(0,h)}{v2,v4}
\end{fmfgraph}}\end{fmffile}}

\newcommand{\phicinq}{\begin{fmffile}{F53}\parbox{12mm}{\begin{fmfgraph}(40,40)
\fmfleft{i}\fmfright{o1,o2}\fmf{plain}{o1,v1,v2,v3,v4,v5,o2}\fmf{plain}{i,v3}
\fmffreeze\fmf{plain}{v1,v5}\fmffreeze\fmf{plain}{v2,v4}
\end{fmfgraph}}\end{fmffile}}

\newcommand{\phisept}{\begin{fmffile}{F54}\parbox{12mm}{\begin{fmfgraph}(40,40)
\fmfleft{i}\fmfright{o1,o2}\fmf{plain}{o1,v1,v2,v3,v4,v5,o2}\fmf{plain}{i,v4}
\fmffreeze\fmf{plain,left,tension=0.2}{v1,v3}\fmffreeze\fmf{plain}{v2,v5}
\end{fmfgraph}}\end{fmffile}}

\newcommand{\phihuit}{\begin{fmffile}{F55}\parbox{14mm}{\begin{fmfgraph}(50,50)
\fmfleft{i}\fmfright{o1,o2}\fmf{plain}{o1,v1,v2,v3,v4,v5,o2}\fmf{plain}{i,v4}
\fmffreeze\fmf{plain,left,tension=0.3}{v2,v3}\fmffreeze\fmf{plain}{v1,v5}
\end{fmfgraph}}\end{fmffile}}

\newcommand{\phidix}{\begin{fmffile}{F57}\parbox{14mm}{\begin{fmfgraph}(50,50)
\fmfleft{i}\fmfright{o1,o2}\fmf{plain}{o1,v1,v2,v3,v4,v5,o2}\fmf{plain}{i,v3}
\fmffreeze\fmf{plain}{v1,v4}\fmffreeze\fmf{plain}{v2,v5}
\end{fmfgraph}}\end{fmffile}}

\newcommand{\phionze}{\begin{fmffile}{F58}\parbox{18mm}{\begin{fmfgraph}(50,50)
\fmfleft{i}\fmfright{o1,o2}\fmf{plain}{o1,v1}\fmf{plain}{v1,w1} \fmf{plain}{w1,v2} \fmf{plain}{v2,w2} \fmf{plain}{w2,v3}
\fmf{plain}{v3,o2}\fmf{plain}{i,v2}\fmffreeze\fmf{plain}{v1,v4}\fmf{plain}{v3,v5}\fmf{plain,left,tension=0.2}{v4,v5,v4}
\end{fmfgraph}}\end{fmffile}}

\newcommand{\phiquatorze}{\begin{fmffile}{F59}\parbox{18mm}{\begin{fmfgraph}(50,50)
\fmfleft{i}\fmfright{o}\fmf{plain}{i,v1}\fmf{plain}{v4,o}\fmf{plain}{v1,v2,v3,v4,v1}\fmf{plain,left}{v2,v3}
\fmffixed{(0.3h,0)}{v2,v3}\fmffixed{(0,.3h)}{v1,v2}\fmffixed{(0,.3h)}{v4,v3}
\end{fmfgraph}}\end{fmffile}}

\input{xy}
\xyoption{all}

\newtheorem{defi}{\indent Definition}
\newtheorem{lemma}[defi]{\indent Lemma}
\newtheorem{cor}[defi]{\indent Corollary}
\newtheorem{theo}[defi]{\indent Theorem}
\newtheorem{prop}[defi]{\indent Proposition}

\newenvironment{proof}{{\bf Proof.}}{\hfill $\Box$}

\newcommand{\bfG}{\mathbf{G}}
\newcommand{\N}{\mathbb{N}}
\newcommand{\Z}{\mathbb{Z}}
\newcommand{\NN}{\mathbb{N}^N_*}
\newcommand{\K}{\mathbb{K}}
\newcommand{\g}{\mathfrak{g}}
\newcommand{\gN}{\mathfrak{g}_{\mathcal{T}}^{(N)}}
\newcommand{\hN}{\mathcal{H}^{D_{M,N}}}
\newcommand{\calB}{\mathcal{B}}
\newcommand{\calE}{\mathcal{E}}
\newcommand{\calHE}{\mathcal{HE}}
\newcommand{\calV}{\mathcal{V}}
\newcommand{\calT}{\mathcal{T}}
\newcommand{\h}{\mathcal{H}}
\newcommand{\fg}{\mathcal{FG}}
\newcommand{\algfg}{\h_{\fg(\calT)}}
\newcommand{\insere}[1]{\overset{#1}{\hookrightarrow}}

\newcommand{\T}{\mathbb{T}}
\newcommand{\F}{\mathbb{F}}

\begin{document}

\title{Mulitgraded Dyson-Schwinger systems}
\date{}
\author{Lo\"\i c Foissy\\ \\
{\small \it Fédération de Recherche Mathématique du Nord Pas de Calais FR 2956}\\
{\small \it Laboratoire de Mathématiques Pures et Appliquées Joseph Liouville}\\
{\small \it Université du Littoral Côte d'Opale-Centre Universitaire de la Mi-Voix}\\ 
{\small \it 50, rue Ferdinand Buisson, CS 80699,  62228 Calais Cedex, France}\\ \\
{\small \it Email: foissy@lmpa.univ-littoral.fr}}

\maketitle

\begin{abstract}
We study systems of combinatorial Dyson-Schwinger equations with an arbitrary number $N$ of coupling constants. 
The considered Hopf algebra of Feynman graphs is $\mathbb{N}^N$-graded,
and we wonder if the graded subalgebra generated by the solution is Hopf or not.

We first introduce a family of pre-Lie algebras which we classify, dually providing systems generating a Hopf subalgebra; 
we also describe the associated groups, as extensions of groups of formal diffeomorphisms on several variables. 

We then consider systems coming from Feynman graphs of a Quantum Field Theory. 
We show that if the number $N$ of independent coupling constants is the number of interactions of the considered QFT, 
then the generated subalgebra is Hopf. For QED, $\varphi^3$ and QCD, we also prove that this is the minimal value of $N$.

All these examples are generalizations of the first family of Dyson-Schwinger systems in the one coupling constant case, called fundamental.
We also give a generalization of the second family, called cyclic.
\end{abstract}

{\bf Keywords.} Dyson-Schwinger systems; Feynman graphs; pre-Lie algebras; combinatorial Hopf algebras.\\

{\bf AMS classification.} 16T05, 81T18, 05C05.

\tableofcontents

\section*{Introduction}

In a Quantum Field Theory (shortly, QFT), the Green functions are developed as a series in the coupling constant, indexed by the set of Feynman graphs. 
These series can be seen at the level of Feynman graphs. They satisfy a certain system $(S)$ of
combinatorial Dyson-Schwinger equation (briefly, SDSE), which uses combinatorial operators of insertion,
and allows to inductively compute the homogeneous components of the Green functions, according to their loop number
\cite{Bergbauer,Broadhurst,Kreimer6,Kreimer2,Kreimer3,Kreimer4,Suijlekom,Kreimer5,Tanasa,Baalen2,Baalen,Yeats}.
Feynman graphs are organized as a Hopf algebra $\h_\fg$, graded by the loop number, and we consider the subalgebra $\h_{(S)}$ of $\h_\fg$
generated by the components of the unique solution of $(S)$.
A natural question is to know if the graded subalgebra generated by the Green functions is Hopf or not.
This problem, and related questions about the nature of the obtained Hopf subalgebras, are the main object of study in \cite{Foissy2,Foissy3,Foissy4,Foissy6}.
It turns out that in the case of QED or $\varphi^3$, which are QFT with only  one interaction, this subalgebra is indeed Hopf;
this is not the case for QCD, with its four interactions. A possibility in this last case is to refine the graduation, or equivalently to
introduce more coupling constants, which makes the subalgebra $\h_{(S)}$ generated by the components of the solution bigger; we shall prove here 
that there exists a $\mathbb{N}^4$-graduation of the Hopf algebra of QCD Feynman graphs, such that  $\h_{(S)}$ is a Hopf subalgebra. \\

The aim of this text is to study SDSE giving a Hopf subalgebra when the Hopf algebra of Feynman graphs is given a $\N^N$-graduation,
generalizing the results of \cite{Foissy3} for the loop number graduation. 
Recall that if we consider only one coupling constant, the Hopf algebra of graphs we consider is $\N$-graded, and we obtained two families of SDSE, 
called fundamental and cyclic, and four operations on SDSE, allowing to obtain all SDSE giving a Hopf subalgebra. 
The graded dual of this Hopf subalgebra is the enveloping algebra of a pre-Lie algebra, described in \cite{Foissy6}.
In the fundamental case, the constant structures of this pre-Lie algebra are polynomial of degree $\leq 1$.  We generalize
this definition to the $\N^N$-graded case (definition \ref{defiprelie}); these objects are called deg1 pre-Lie algebras.
Their classification is done in theorem \ref{theoclassif}. 
As enveloping algebras of free pre-Lie algebras are Grossman-Larson Hopf algebras \cite{Grossman1,Grossman2}, 
dually the enveloping algebra of a deg1 pre-Lie algebra  can be embedded in a Connes-Kreimer Hopf algebra of decorated rooted trees
\cite{Connes,Foissy1}, giving in this way a family of SDSE such the associated subalgebra
is Hopf (theorem \ref{theo14}). We also describe the group associated to such pre-Lie algebras; they all contain a group of formal diffeomorphisms.
  
We then proceed to SDSE coming from a QFT.
We first study all the possible graduations of $\h_\fg$ which are defined from combinatorial datas associated to Feynman graphs, 
such as the number of vertices, of internal or external half-edges or edges, or the external structure: 
we prove that such a $\N^N$-graduation  is associated to a matrix  $C\in M_{N,|\calV|}(\mathbb{Q})$, where $\calV$ is the set of possible vertices 
in the Feynman graphs of the theory (proposition \ref{propgraduation}); the rank of $C$ is of special importance here.
We show how to lift these systems at the level of decorated rooted trees, using a universal property, and we recover in this way
SDSE associated to deg1 pre-Lie algebras previously described, if the rank of $C$ is the cardinality of $\calV$.
We may ask the question of the minimal rank of $C$ required to obtain a Hopf subalgebra: it is smaller than $|\calV|$.
In QED or $\varphi^n$, as this cardinality is $1$, the answer is obviously $1$; for QCD, we prove  in proposition \ref{propQCD} that it is also
$|\calV|=4$. The main idea is to produce primitive Feynman graphs with an arbitrarily large number of vertices of any kind,
and we conjecture that for any QFT with enough primitive Feynman graphs,
the minimal rank of the graduation is the number of interactions of the theory.
We shall conclude with a generalization of the second family of SDSE in the $\N$-graded case, namely cyclic SDSE.\\

This article is organized as follows. The first section contains reminders on Connes-Kreimer Hopf algebras of decorated rooted trees,
their universal properties, their graduations and their graded duals. In the second section, we introduce the notion of combinatorial SDSE in
Connes-Kreimer Hopf algebras; we give three operations on SDSE, and also study the effect of changing the graduation of the subalgebra $\h_{(S)}$
generated by the unique solution of such a SDSE. We then introduce and classify deg1 pre-Lie algebras in the next section, which dually give
us a first family of $\N^N$-graded SDSE. The group associated to these pre-Lie algebras are described in the fourth section.
Feynman graphs of a given QFT, their Hopf-algebraic structure and their SDSE are introduced and studied in the next section.
The last, independent, section deals with a generalization of cyclic SDSE.\\

{\bf Aknowledgment.} The research leading these results was partially supported by the French National Research Agency under the reference
ANR-12-BS01-0017.\\

{\bf Notations.} \begin{enumerate}
\item Let $M$ and $N$ be nonnegative integers. We denote by $[M]$ the set of integers $\{1,\ldots,M\}$ and by $\NN$ the set of nonzero elements of $\N^N$.
\item The canonical basis of $\K^N$ (and of $\mathbb{Z}^N$) is denoted by $(\epsilon_1,\ldots,\epsilon_N)$.
\item Let $a,b \in \K$. We denote by $F_{a,b}(X)$ the formal series:
$$F_{a,b}(X)=\sum_{k=0}^\infty \frac{a(a-b)\ldots (a-b(k-1))}{k!}X^k=\begin{cases}
(1+bX)^{\frac{a}{b}}\mbox{ if }b\neq 0,\\
e^{aX} \mbox{ if }b=0.
\end{cases}$$
Note that for all $a,a',b \in \K$, $F_{a+a',b}(X)=F_{a,b}(X)F_{a',b}(X)$.
\end{enumerate}

\section{Hopf algebras of decorated trees}

Let us start with a few reminders on the Connes-Kreimer Hopf algebras of decorated trees \cite{Connes,Foissy1} and related algebraic structures.
We consider a nonempty set $D$, which we call the set of decorations.

\subsection{Definition and universal property}

\begin{defi}
\begin{enumerate}
\item A tree is a finite graph, connected, with no loop; a rooted tree is a tree with a pointed vertex, called the root;
a rooted tree decorated by $D$ is a pair $(T,d)$, where $T$ is a rooted tree and $d$ is a map from the set $V(T)$ of vertices of $T$ to $D$;
for all $v \in V(T)$, $d(v)$ is called the decoration of $v$. The set of isoclasses of rooted trees decorated by $D$ is denoted by $\T^D$.
\item The algebra $\h^D$ of rooted trees decorated by $D$ is the free commutative associative algebra generated by $\T^D$. 
By definition, the set $\F^D$ of rooted forests decorated by $D$, that is to say  monomials in $\T^D$,
or finite disjoint unions of elements of $\T^D$, is a basis of $\h^D$. The product of $\h^D$ is the disjoint union of decorated rooted forests.
\end{enumerate}\end{defi}

{\bf Examples.} We draw rooted trees with their root at the bottom. 
\begin{enumerate}
\item The rooted trees decorated by $D$ with $n \leq 4$ vertices are:
$$\tdun{$a$},\: a\in D;\hspace{1cm} \tddeux{$a$}{$b$},\: (a,b)\in D^2;\hspace{1cm}
\tdtroisun{$a$}{$c$}{$b$}=\tdtroisun{$a$}{$b$}{$c$},\: \tdtroisdeux{$a$}{$b$}{$c$},\:(a,b,c)\in D^3;$$
$$ \tdquatreun{$a$}{$d$}{$c$}{$b$}=\tdquatreun{$a$}{$c$}{$d$}{$b$}=\ldots
=\tdquatreun{$a$}{$b$}{$c$}{$d$},\: \tdquatredeux{$a$}{$d$}{$b$}{$c$}=\tdquatretrois{$a$}{$b$}{$c$}{$d$},\:
 \tdquatrequatre{$a$}{$b$}{$d$}{$c$}= \tdquatrequatre{$a$}{$b$}{$c$}{$d$},\: \tdquatrecinq{$a$}{$b$}{$c$}{$d$},\:(a,b,c,d)\in D^4.$$
\item  The rooted forests decorated by $D$ with $n\leq 3$ vertices are:
$$1;\hspace{1cm}\tdun{$a$},\: a\in D;\hspace{1cm} \tddeux{$a$}{$b$}, \:\tdun{$a$}\tdun{$b$}=\tdun{$b$}\tdun{$a$},\: (a,b)\in D^2;$$
$$\tdtroisun{$a$}{$c$}{$b$}=\tdtroisun{$a$}{$b$}{$c$},\: \tdtroisdeux{$a$}{$b$}{$c$},\:
\tddeux{$a$}{$b$}\tdun{$c$}=\tdun{$c$}\tddeux{$a$}{$b$},\:
\tdun{$a$}\tdun{$b$}\tdun{$c$}=\tdun{$a$}\tdun{$c$}\tdun{$b$}=\ldots=\tdun{$c$}\tdun{$b$}\tdun{$a$},\:(a,b,c)\in D^3.$$
\end{enumerate}

The algebra $\h^D$ can also be defined by a universal property \cite{Connes,Moerdijk2}:

\begin{prop}
Let $d \in D$. The linear endomorphism $B_d$ of $\h^D$ sends any rooted forest $F \in \F^D$ to $B_d(F)\in \T^D$
obtained in grafting the different trees of $F$ on a common root decorated by $d$.
This family of endomorphisms satisfy the following universal property: if $A$ is a commutative algebra, and for all $d \in D$, $L_d:A\longrightarrow A$
is a linear endomorphism, there exists a unique algebra morphism $\phi:\h^D\longrightarrow A$ such that for all $d\in D$, $\phi \circ B_d=L_d\circ \phi$.
\end{prop}

{\bf Example.} If $a,b,c,d \in D$, $B_a(\tdun{$b$}\tddeux{$c$}{$d$})=\tdquatretrois{$a$}{$c$}{$d$}{$b$}$.\\

This universal property can be used to define the Connes-Kreimer coproduct of $\h^D$:

\begin{prop} \begin{enumerate}
\item There exists a unique coproduct on $\h^D$ such that for all $d\in D$, for all $x\in \h^D$:
$$\Delta\circ B_d(x)=B_d(x)\otimes 1+(Id\otimes B_d)\circ \Delta(x).$$
With this coproduct, $\h^D$ becomes a Hopf algebra. Its counit is the map:
$$\varepsilon:\left\{\begin{array}{rcl}
\h^D&\longrightarrow&\K\\
F\in \F^D&\longrightarrow&\delta_{F,1}.
\end{array}\right.$$
\item Let $A$ be a commutative Hopf algebra, and for all $d\in D$, let $L_d:A\longrightarrow A$ a linear endomorphism such that for all $x\in A$:
$$\Delta\circ L_d(x)=L_d(x)\otimes 1+(Id \otimes L_d)\circ \Delta(x).$$
The unique algebra morphism $\phi:\h^D\longrightarrow A$ such that for all $d\in D$, $\phi \circ B_d=L_d\circ \phi$ is a Hopf algebra morphism.
\end{enumerate}\end{prop}

This coproduct admits a combinatorial description in terms of admissible cuts. For example, if $a,b,c,d \in D$:
$$\Delta \tdquatredeux{$d$}{$c$}{$b$}{$a$}=\tdquatredeux{$d$}{$c$}{$b$}{$a$} \otimes 1+1\otimes \tdquatredeux{$d$}{$c$}{$b$}{$a$}
+\tddeux{$b$}{$a$} \otimes \tddeux{$d$}{$c$}+\tdun{$a$}\otimes \tdtroisun{$d$}{$c$}{$b$}+
\tdun{$c$} \otimes \tdtroisdeux{$d$}{$b$}{$a$} +\tddeux{$b$}{$a$}\tdun{$c$}\otimes \tdun{$d$}
+\tdun{$a$}\tdun{$c$}\otimes \tddeux{$d$}{$b$}.$$

Here is another application of the universal property:

\begin{prop}\label{propmorphismes}
Let $a=(a_d)_{d\in D}$ be a family of elements of $\K$. We denote by $\phi_a$ the unique Hopf algebra endomorphism of $\h^D$ such that 
for all $d\in D$, $\phi \circ B_d=a_dB_d\circ \phi$. For any forest $F\in \F^D$, denoting by $V(F)$ the set of vertices of $F$:
$$\phi_a(F)=\left(\prod_{v\in V(F)} a_{d(v)} \right)F.$$
Consequently, if for all $d\in D$, $a_d\neq 0$, $\phi_a$ is an automorphism.
\end{prop}

\begin{proof} We consider the endomorphism $\varphi$ defined by:
$$\forall F\in \F^D,\:\varphi(F)=\left(\prod_{v\in V(F)} a_{d(v)} \right)F.$$
Let $F,F_1,F_2\in \F^D$. As $V(F_1F_2)=V(F_1)\sqcup V(F_2)$, $\varphi(F_1F_2)=\varphi(F_1)\varphi(F_2)$, 
$\varphi$ is an algebra endomorphism.  As $V(B_d(F))=V(F)\sqcup\{root(F)\}$, $\varphi(B_d(F))=a_dB_d(\varphi(F))$. 
Consequently, $\varphi \circ B_d=a_dB_d\circ \varphi$. By unicity in the universal property, $\varphi=\phi_a$. \end{proof}

\subsection{Graduation and duality}

\begin{defi} \begin{enumerate}
\item  A $\N^N$-graded set is a pair $(D,deg)$, where $D$ is a set and $deg:D\longrightarrow \N^N$ is a map. For all $\alpha \in \N^N$,
we put $D_\alpha=deg^{-1}(\alpha)$. We shall say that the $\N^N$-graded $D$ is connected if $D_0=\emptyset$ and if for all $\alpha \in \N^N$,
$deg^{-1}(\alpha)$ is finite.
\item Let $D$ be a $\N^N$-graded connected set. For all forest $F\in \F^D$, we put:
$$deg(F)=\sum_{v\in V(F)} deg(d(v)).$$
This induces a connected $\N^N$-graduation of the Hopf algebra $\h^D$, with:
$$\forall \alpha \in \N^N,\: (\h^D)_\alpha=Vect(F\in \F^D\mid deg(F)=\alpha).$$ 
Moreover, for this graduation, $B_d$ is homogeneous of degree $deg(d)$ for all  $d\in D$.
\end{enumerate}\end{defi}

If $D$ is a $\N^N$-graded connected set, then, as $\h^D$ is a graded connected Hopf algebra, its graded dual $(\h^D)^*$ is also a Hopf algebra
\cite{Hoffman,Panaite}. As a vector space, it can be identified with $\h^D$, by the help of the symmetric pairing defined by:
$$\forall F,G \in \F^D,\: \langle F,G\rangle=s_F \delta_{F,G},$$
where $s_F$ is the number of symmetries of $F$. The coproduct $\Delta'$ of $(\h^D)^*$ is given by:
$$\forall T_1,\ldots,T_k\in \T^D,\: \Delta'(T_1\ldots T_k)=\sum_{I\subseteq [k]} \left(\prod_{i\in I} T_i\right)\otimes \left(\prod_{i\notin I} T_i\right).$$
Its product $\star$ is given by graftings: this is the Grossman-Larson product \cite{Grossman1,Grossman2,Grossman3}.
For example:
$$\tddeux{$a$}{$b$}\star\tddeux{$c$}{$d$}=\tddeux{$a$}{$b$}\tddeux{$c$}{$d$}+\tdquatredeux{$c$}{$d$}{$a$}{$b$}+\tdquatrecinq{$c$}{$d$}{$a$}{$b$}.$$
Note that this graded dual does not depend of the choice of the connected graduation of $D$. \\

By the Cartier-Quillen-Milnor-Moore's theorem, $(\h^D)^*$ is the enveloping algebra of a Lie algebra $\g^D$. By construction of the coproduct $\Delta'$,
the set $\T^D$ is a basis of $\g^D$; by definition of the Grossman-Larson product, for all $T,T' \in \T^D$:
$$[T,T']=\sum_{v'\in V(T')} \mbox{grafting of $T$ on $v'$}-\sum_{v\in V(T)} \mbox{grafting of $T'$ on $v$}.$$
We define a product $*$ on $\g^D$ by:
$$T*T'=\sum_{v'\in V(T')} \mbox{grafting of $T$ on $v'$}.$$
For any $x,y\in \g^D$, $[x,y]=x*y-y*x$. For example:
\begin{align*}
\tdun{$c$}*\tddeux{$a$}{$b$}&=\tdtroisun{$a$}{$b$}{$c$}+\tdtroisdeux{$a$}{$b$}{$c$},&
\tddeux{$a$}{$b$}*\tdun{$c$}&=\tdtroisdeux{$c$}{$a$}{$b$}.
\end{align*}
 This product is not associative, but is pre-Lie:

\begin{defi}
A (left) pre-Lie algebra is a pair $(V,*)$, where $V$ is a vector space and $*$ is a bilinear product on $V$, such that for all $x,y,z\in V$:
$$(x*y)*z-x*(y*z)=(y*x)*z-y*(x*z).$$
If $(V,*)$ is pre-Lie, the bracket defined by $[x,y]=x*y-y*x$ is a Lie bracket.
\end{defi}

Moreover, Chapoton and Livernet proved, using the theory of operads, that $\g^D$ is a free pre-Lie algebra \cite{Chapoton1,Chapoton2}:

\begin{theo}
Let $A$ be a pre-Lie algebra and let $a_d\in A$ for all $d\in D$.
There exists a unique pre-Lie algebra morphism $\phi:\g^D\longrightarrow A$ such that $\phi(\tdun{$d$})=a_d$ for all $d\in D$. 
In other words, $\g^D$ is, as a pre-Lie algebra, freely generated by the elements $\tdun{$d$}$, $d\in D$.
\end{theo}

\subsection{Completion}

We graduate $\h^D$ by the number of vertices of forests, that is to say we consider the graduation induced by the map $deg:D\longrightarrow \N$,
sending every element of $D$ to $1$. This graduation induces a distance $d$ on $\h^D$, defined by:
$$d(f,g)=2^{-val(f-g)}.$$
The metric space $\h^D$ is not complete: its completion is denoted by $\widehat{\h^D}$. As a vector space, it is the space of commutative formal series
in $\T^D$. The product of $\h^D$, being homogeneous of degree $0$, is continuous, so can be extended to $\widehat{\h^D}$:
this gives the usual product of formal series. Similary, for any $d\in D$, $B_d$, being homogeneous of degree $1$, is continuous so can be extended
to a map $B_d:\widehat{\h^D}\longrightarrow \widehat{\h^D}$.

\section{Multigraded SDSE}

\subsection{Definition}

\begin{defi}\label{defiprelie}
Let $D=D_1\sqcup \ldots \sqcup D_M$ be a partitioned set.
Let $(f_d)_{d\in D}$ be a family of elements of $\K\langle\langle x_1,\ldots,x_M\rangle\rangle$. The system of Dyson-Schwinger equations
(briefly, SDSE) associated to these elements is:
$$\forall i\in [M],\: X_i=\sum_{d\in D_i} B_d(f_d(X_1,\ldots,X_M)),$$
where $X=(X_1,\ldots,X_M)$ belongs to $\widehat{\h^D}^M$.
\end{defi}

By convenience, we generally index the family of unknows by $[M]$, but it is of course possible to index them by any finite set.

\begin{prop}
Let $(S)$ be a SDSE. It has a unique solution.
\end{prop}

\begin{proof} If $X=(X_1,\ldots,X_M)$  is a solution of $(S)$, then for all $i$, $X_i$ is a infinite span of trees, so belongs to the augmentation ideal $\h_+^D$.
Hence, it is enough to prove that $(S)$ has a unique solution in $\widehat{\h^D_+}^M$.
Let us consider the following map:
$$\Theta:\left\{\begin{array}{rcl}
\widehat{\h^D}^M\hspace{-3mm}&\longrightarrow&\widehat{\h^D}^M\\
(X_1,\ldots,X_M)&\longrightarrow&\left( \displaystyle\sum_{d\in D_i} B_i(f_d(X_1,\ldots,X_M))\right)_{i\in [M]}.
\end{array}\right.$$
As $B_d$ is homogeneous of degree $1$ for all $d$, we obtain that for all $f,g \in \widehat{\h^D}^M$:
$$d(\Theta(f),\Theta(f))\leq \frac{1}{2} d(f,g).$$
So $\Theta$ is a contracting map. As $\widehat{\h^D}^M$ is complete, $\Theta$ has a unique fixed point $(X_1,\ldots,X_M)$,
which is the unique solution of $(S)$. \end{proof} \\

{\bf Remarks.} \begin{enumerate}
\item As the $D_i$ are disjoint, the nonzero $X_i$ are sum of trees with roots decorated by elements of $D_i$, so are algebraically independent.
\item If $X_i=0$, we can delete the $i$-th equation of $(S)$ and replace $f_d$ by $(f_d)_{\mid x_i=0}$ for all $d\in D$, without changing $\h_{(S)}$.
\end{enumerate}

{\bf We now assume that all the $X_i$ are nonzero}  (and, as a consequence, are algebraically independent).

\begin{defi}
Let $D$ be a connected $\N^N$-graded set, inducing a  connected $\N^N$-graduation of the Hopf algebra $\h^D$. 
Let $(S)$ be a SDSE on $D$. 
\begin{enumerate}
\item The unique solution of $S$ is denoted by $X=(X_1,\ldots,X_M)$, and the homogeneous components of $X_i$ 
are denoted by $X_i(\alpha)$, $i\in [M]$, $\alpha\in \N^N$. 
\item The subalgebra of $\h^D$ generated by the $X_i(\alpha)$'s is denoted by $\h_{(S)}$. 
\item We shall say that $(S)$ is Hopf if $\h_{(S)}$ is a Hopf subalgebra of $\h^D$.
\end{enumerate}\end{defi} 

Note that $\h_{(S)}$ depends on the choice of the graduation. \\

{\bf Example.} Here is an example of SDSE. Le us fix $k\geq 1$ and $d_0,\ldots,d_k\in \mathbb{N}$. 
For any $\alpha=(\alpha_0,\ldots,\alpha_k)\in [N]^{k+1}$, we put:
$$deg(\alpha)=d_0\epsilon_{\alpha_0}+\ldots+d_k\epsilon_{\alpha_k}\in \mathbb{Z}^N.$$
The set of decorations is:
$$D=\{\alpha\in [N]^{k+1}\mid deg(\alpha)\in \N^N\setminus\{0\}\}.$$
The Hopf algebra $\h^D$ inherits a connected $\N^N$-graduation. We consider the SDSE:
\begin{align}
\label{system}
(S)_{FdB}:\: \forall i\in [N],\:X_i=\sum_{\alpha\in [N]^k} B_{(i,\alpha)}\left((1+X_{\alpha_1})^{d_1}\ldots (1+X_{\alpha_k})^{d_k}(1+X_i)^{d_0+1}\right).
\end{align}
In particular, if $(d_0,\ldots,d_k)=(0,1,\ldots,1)$, this gives:
\begin{align*}
\forall i\in [N],\:X_i=\sum_{\alpha\in [N]^k} B_{(i,\alpha)}\left((1+X_{\alpha_1})\ldots (1+X_{\alpha_k})\right).
\end{align*}
Taking $k=2$, the components of $X$ are a commutative version of the elements of Definition 20 in \cite{Foissy2}, which generate a Hopf algebra
isomorphic to the free Faà di Bruno Hopf algebra on $N$ variables. We shall prove that it is indeed a Hopf SDSE, related to the Faà di Bruno Hopf algebra
on $N$ variables.

\subsection{Simplification of the hypotheses}

\begin{lemma}
Let $(S)$ be a Hopf SDSE, and let $d\in D$. If $f_d(0,\ldots,0)=0$, then $f_d=0$. 
\end{lemma}

\begin{proof} Let $i\in [M]$, such that $d\in D_i$. As $f_d(0,\ldots,0)=0$, $\tdun{$d$}$ does not appear in $X_i$, 
and $\tdun{$d$}$ never appears in any element of $\h_{(S)}$. Let us assume that $f_d\neq 0$. As the $X_j$ are algebraically independent, 
$f_d(X_1,\ldots,X_N)\neq 0$, and there exists a linear form $g$ on $\widehat{\h^D} $, such that $g(f_d(X_1,\ldots,X_N))=1$.
Then $(g\otimes Id)\circ \Delta(X_i)$ is an element of $\h_{(S)}$, where the term $g(f_d(X_1,\ldots,X_N))\tdun{$d$}=\tdun{$d$}$ appears:
this is a contradiction. So $f_d=0$. \end{proof} \\
 
Consequently, if $\h_{(S)}$ is Hopf and $f_{d_0}(0,\ldots,0)=0$ for a certain $d_0\in D_i$, we can rewrite the $i$-th equation of $(S)$ in the following way:
$$X_i=\sum_{d\in D_i\setminus \{d_0\}} B_d(f_d(X_1,\ldots,X_M)).$$

{\bf We now assume that for all $d$, $f_d(0,\ldots,0)\neq 0$}.

\begin{lemma}
We consider the two SDSE:
\begin{align*}
(S):\forall i\in [M], \: X_i&=\sum_{d\in D_i} B_d(f_d(X_1,\ldots,X_M)),\\
(S'):\forall i\in [M], \: Y_i&=\sum_{d\in D_i} B_d\left(\frac{f_d(Y_1,\ldots,Y_M)}{f_d(0,\ldots,0)}\right).
\end{align*}
For all $d \in D$, we put $a_d=f_d(0,\ldots,0)$. Let $\phi_a$ be the Hopf algebra isomorphism defined in proposition \ref{propmorphismes}.
Then for all $i\in [M]$, $X_i=\phi_a(Y_i)$; $\h_{(S)}=\phi_a(\h_{(S')})$ and $(S)$ is Hopf, if and only if, $(S')$ is Hopf.
\end{lemma}

\begin{proof} We put:
$$g_d(x_1,\ldots,x_M)=\frac{f_d(x_1,\ldots,x_M)}{f_d(0,\ldots,0)}.$$
As $\phi_a \circ B_d=f_d(0,\ldots,0) B_d\circ \phi_a$ for all $d$, we obtain:
\begin{align*}
\phi_a(Y_i)&=\sum_{d\in D_i} \phi_a\circ B_d(g_d(Y_1,\ldots,Y_M))\\
&=\sum_{d\in D_i} f_d(0,\ldots,0)B_d\circ \phi_a(g_d(Y_1,\ldots,Y_M))\\
&=\sum_{d\in D_i} f_d(0,\ldots,0)B_d( g_d(\phi_a(Y_1),\ldots,\phi_a(Y_M)))\\
&=\sum_{d\in D_i} B_d( f_d(\phi_a(Y_1),\ldots,\phi_a(Y_M))).
\end{align*}
So $(\phi_a(Y_1),\ldots,\phi_a(Y_M))$ is the unique solution of $(S)$. \end{proof}\\

{\bf  We now assume that $f_d(0,\ldots,0)=1$ for all $d\in D$.}

\begin{lemma}\label{lemmaseriesidentiques}
Let $(S)$ be a Hopf SDSE, $d_1,d_2$ be two elements in the same $D_i$, of the same degree. Then $f_{d_1}=f_{d_2}$.
\end{lemma}

\begin{proof} Let us denote by $\alpha$ the common degree of $d_1$ and $d_2$. The homogeneous component of $X_i$ of degree $\alpha$
has the form $\tdun{$d_1$}\:+\tdun{$d_2$}\:+\ldots$; consequently, if we consider the linear forms:
\begin{align*}
f_1:&\left\{\begin{array}{rcl}
\h^D&\longrightarrow&\K\\
F\in \F^D&\longrightarrow&\delta_{F,\tdun{$d_1$}\:},
\end{array}\right.&
f_2:&\left\{\begin{array}{rcl}
\h^D&\longrightarrow&\K\\
F\in \F^D&\longrightarrow&\delta_{F,\tdun{$d_2$}\:},
\end{array}\right.&
\end{align*}
then the restriction of $f_1$ and $f_2$ to $\h_{(S)}$ are equal. As $\h_{(S)}$ is Hopf:
\begin{align*}
f_{d_1}(X_1,\ldots,X_M)=(Id \otimes f_1)\circ \Delta(X_i)&=(Id \otimes f_2)\circ \Delta(X_i)=f_{d_2}(X_1,\ldots,X_M).
\end{align*}
So $f_{d_1}=f_{d_2}$. \end{proof}\\

Note that, if the SDSE is Hopf, we can write it under the form:
$$\forall i\in [M],\: X_i=\sum_{\alpha\in \N^N_*}\underbrace{\left(\sum_{i\in D_i, deg(i)=\alpha} B_i\right)}_{=B_{i,\alpha}}(f_\alpha(X_1,\ldots,X_M))
=\sum_{\alpha\in \N^N_*} B_{i,\alpha}(f_\alpha(X_1,\ldots,X_M)).$$

\subsection{Operations on SDSE}

\begin{defi}
Let $D=D_1\sqcup \ldots \sqcup D_M$ be a $\N^N$-graded connected partitioned set. We consider the SDSE given by:
$$(S):\:\forall i\in [M], \:X_i=\sum_{d\in D_i} B_d(f_d(X_1,\ldots,X_M)).$$
\begin{enumerate}
\item{\bf (Change of variables)}
Let $a=(a_1,\ldots,a_M)$ be a family of  nonzero scalars. The SDSE obtained from $(S)$ by the change of variables associated to these coefficients is:
$$(S)_a:\:\forall i\in [M], \:Y_i=\sum_{d\in D_i} B_d(f_d(a_1Y_1,\ldots,a_MY_M)).$$
\item{\bf (Restriction)} 
Let $I\subseteq [M]$. The restriction of $(S)$ to $I$ is the SDSE  given by:
$$(S)_{\mid I}:\:\forall i\in I, \:X_i=\sum_{d\in D_i} B_d(g_d(X_j,j\in I)),$$
where for all $d\in I$, $g_d={f_d}_{\mid x_j=0\: \mbox{\scriptsize for all } j\notin I}\in \K[[X_j,j\in I]]$.
\end{enumerate}\end{defi}

\begin{prop}\label{proprestriction}\begin{enumerate}
\item Let $(S)$ be a SDSE and let $(S)_a$ be another SDSE, obtained from $(S)$ by a change of variables. 
We define the coefficients $a_d$, $d\in D$, by:
$$a_d=a_i \mbox{ if }d\in D_i.$$
Let $\phi_a$ be the Hopf algebra isomorphism defined in proposition \ref{propmorphismes}. The unique solution of $(S)_a$ is:
$$\left(\frac{1}{a_1}\phi_a(X_1),\ldots,\frac{1}{a_M}\phi_a(X_M)\right).$$
Hence, $\h_{(S)_a}=\phi_a(\h_{(S)})$ and $(S)$ is Hopf if, and only if $(S)_a$ is Hopf.
\item Let $I \subseteq M$. We define the coefficients $a_d$, $d\in D$, by:
$$a_d=\begin{cases}
\displaystyle 1\mbox{ if }d\in \bigsqcup_{i\in I} D_i,\\
0\mbox{ otherwise}.
\end{cases}$$
Let $\phi_a$ be the Hopf algebra morphism defined in proposition \ref{propmorphismes}. The unique solution of $(S)_{\mid I}$ is:
$$\left(\phi_a(X_i)\right)_{i\in I}.$$
Hence, $\h_{(S)_{\mid I}}=\phi_a(\h_{(S)})$ and, if $(S)$ is a Hopf SDSE, then $(S)_{\mid I}$ is also a Hopf SDSE.
\end{enumerate}\end{prop}

\begin{proof} 1. For all $i \in [M]$, we put $Y_i=\frac{1}{a_i}\phi_a(X_i)$. Then:
\begin{align*}
Y_i&=\frac{1}{a_i}\sum_{d\in D_i} \phi_a\circ B_d(f_d(X_1,\ldots,X_M))\\
&=\sum_{d\in D_i} B_d \circ \phi_a(f_d(X_1,\ldots,X_M))\\
&=\sum_{d\in D_i} B_d (f_d(\phi_a(X_1),\ldots,\phi_a(X_M)))\\
&=\sum_{d\in D_i} B_d (f_d(a_1Y_1,\ldots,a_NY_M)).
\end{align*}
So $Y=(Y_1,\ldots,Y_M)$ is the solution of $(S)_a$. \\

2. Proved in a similar way, noting that $\phi_a(X_i)=Y_i$ if i $\in I$ and $0$ otherwise. \end{proof}

\begin{defi}[Concatenation]
Let $(S)$ and $(S')$ be two SDSE, respectively associated to partitioned $\N^N$-graded sets $D=D_1\sqcup \ldots \sqcup D_M$
and $D'=D'_1\sqcup \ldots \sqcup D'_{M'}$, and to formal series $(f_d)_{d\in [M]}$ and $(f'_d)_{d\in [M']}$. The concatenation of $(S)$ and $(S')$
is the system associated to the $\N^N$-graded partitioned set 
$D\sqcup D'=D_1\sqcup \ldots \sqcup D_M\sqcup D'_1\sqcup \ldots \sqcup D'_{M'}$ given by:
$$(S) \sqcup (S'):\left\{\begin{array}{rl}
\mbox{ if } 1\leq i \leq M,\:, X_i&\displaystyle=\sum_{d\in D_i} B_d(f_d(X_1,\ldots,X_M)),\\[5mm]
\mbox{ if } M+1 \leq i \leq M+M',\:X_i&\displaystyle=\sum_{d\in D'_{i-M}} B_d(f'_d(X_{M+1},\ldots,X_{M+M'})).
\end{array}\right.$$
\end{defi}

\begin{prop}
Let $(S)$ and $(S')$ be two SDSE. Then $(S)\sqcup (S')$ is Hopf if, and only if, $(S)$ and $(S')$ are Hopf.
\end{prop}

\begin{proof}
$\Longrightarrow$. Let us assume that $(S) \sqcup (S')$ is Hopf. Then $(S)\sqcup (S')_{\mid [M]}=(S)$ and, up to a reindexation,
$(S) \sqcup (S')_{\mid [M+M']\setminus [M]}=(S')$. By proposition \ref{proprestriction}, $(S)$ and $(S')$ are Hopf. \\

$\Longleftarrow$. Let us assume that $(S)$ and $(S')$ are Hopf. Then $\h_{(S) \sqcup (S')}$ is isomorphic to $\h_{(S)}\otimes \h_{(S')}
\subseteq \h^D \otimes \h^{D'} \subseteq \h^{D\sqcup D'}$. As $\h_{(S)}$ and $\h_{(S')}$ are Hopf subalgebras of $\h^D$ and $\h^{D'}$,
$\h_{(S)}\otimes \h_{(S')}$ is a Hopf subalgebra of $\h^{D\sqcup D'}$, so $(S) \sqcup (S')$ is Hopf. \end{proof}\\

{\bf Remark.} As in \cite{Foissy3}, it is possible to define an operation of dilatation for multigraded SDSE.
We will not use it here.

\subsection{Changes of graduation}

\label{sectgrad}

Let $D$ be a $\N^N$-graded connected set. Let $C\in \mathcal{M}_{N',N}(\mathbb{Q})$. We assume the following hypothesis: 
if $\alpha \in \N^N$ satisfies $D_\alpha \neq(0)$, then $C\alpha \in \N^{N'}_*$. We give $D$ a $\N^{N'}$-graduation by:
$$D'_\beta=\bigsqcup_{\alpha \in \N^N, C\alpha=\beta} D_\alpha.$$
This defines another connected graduation of $D$. Consequently, $\h^D$ inherits a second graduation:
$$\h^D_{(\beta)'}=\bigoplus_{\alpha,C\alpha=\beta} \h_{(\alpha)}^D.$$
Let $(S)$ be a SDSE on $D$. The solution $X$ of $(S)$ can be decomposed into two ways:
$$X_i=\sum_{\alpha \in \N^N} X_i(\alpha)=\sum_{\beta\in \N^{N'}} X_i'(\beta).$$
Hence, we obtain two subalgebras, denoted by $\h_{(S)}$ and $\h'_{(S)}$. 

\begin{lemma} \label{lemmagrad}
Under the preceding hypotheses:
\begin{enumerate}
\item $\h_{(S)}'\subseteq \h_{(S)}$; if $Ker(C)=(0)$, this is an equality.
\item If $\h_{(S)}'$ is Hopf, then $\h_{(S)}$ is Hopf.
\end{enumerate}
\end{lemma}

\begin{proof} Let $\beta \in \N^{N'}$. Then:
$$X_i'(\beta)=\sum_{C\alpha=\beta} X_i(\alpha).$$
Hence, $\h_{(S)}'\subseteq \h_{(S)}$. Let us assume that $Ker(C)=(0)$. Let $\alpha\in \N^N$. We put $\beta=C\alpha$. As $C$ is injective,
$X_i'(\beta)=X_i(\alpha)$, so $X_i(\alpha)\in \h_{(S)}'$, and finally $\h_{(S)}=\h_{(S)}'$. \\

Let us assume that $\h_{(S)}'$ is Hopf. We denote by $\pi_\alpha$ the canonical projection on $\h^D(\alpha)$.
For all $\beta\in \N^{N'}$:
$$\pi_\alpha(X'_i(\beta))=\begin{cases}
X_i(\alpha)\mbox{ if }C\beta=\alpha,\\
0\mbox{ otherwise}. 
\end{cases}$$
Moreover, for all $x,y\in \h^D$:
$$\pi_\alpha(xy)=\sum_{\alpha'+\alpha''=\alpha}\pi_{\alpha'}(x)\pi_{\alpha''}(y).$$
This implies that for all $\alpha \in \N^N_*,$, $\pi_\alpha\left(\h_{(S)}'\right)\subseteq \h_{(S)}$. For $\beta=C\alpha$:
\begin{align*}
\Delta(X_i(\alpha))&=\Delta\circ \pi_\alpha(X_i'(\beta))\\
&=\sum_{\alpha'+\alpha''=\alpha} (\pi_{\alpha'} \otimes \pi_{\alpha''})\circ \Delta(X_i(\beta))\\
&\in \sum_{\alpha'+\alpha''=\alpha}\pi_{\alpha'}\left(\h'_{(S)}\right) \otimes \pi_{\alpha''}\left(\h'_{(S)}\right)\\
&\in \h_{(S)}\otimes \h_{(S)}.
\end{align*}
So $\h_{(S)}$ is a Hopf subalgebra of $\h^D$. \end{proof} \\

We shall often restrict ourselves to matrices $C$ whose rank is $N'$. 
One natural question is to find the smallest $N$ such that there exists a $\N^{N}$-graduation making the studied SDSE Hopf. \\

\section{A family of pre-Lie algebras}

If $(S)$ is a Hopf SDSE, as $X_i$ is an infinite span of trees with roots decorated by $D_i$.
Moreover, in $\h_{(S)}$, any linear span of rooted trees with roots decorated by $D_i$ is a linear span of $X_i(\alpha)$; 
hence, we can write the coproduct of $X_i$ under the form:
$$\Delta(X_i)=X_i\otimes 1+\sum_{\alpha \in \NN} P_{i,\alpha}(X_1,\ldots,X_n)\otimes X_i(\alpha).$$
So $\h_{(S)}$ is a commutative combinatorial Hopf algebra in the sense of \cite{Loday}. 
Hence, its dual is the enveloping of algebra of a pre-Lie algebra $\g_{(S)}$. It is generated by the elements $f_i(\alpha)$, dual to the nonzero $X_i(\alpha)$;
for all $i,j \in [M]$, for all $\alpha,\beta \in \N^N_*$, there exists a scalar $\lambda_{i,j}(\alpha,\beta)$, such that:
$$f_j(\beta)*f_i(\alpha)=\lambda_{i,j}(\alpha,\beta) f_i(\alpha+\beta),$$
where $*$ is the pre-Lie product of $\g_{(S)}$. 
When $N=1$, if the system is fundamental, we proved in \cite{Foissy6}  that these coefficients are polynomial of degree $\leq 1$. 
We here generalize this case for any $N$.

\subsection{Definition and examples}

\begin{defi}
Let $(\g,*)$ be a pre-Lie algebra. We shall say that it is deg1 if there exists a basis $(f_i(\alpha))_{i\in [M], \alpha \in \NN}$ of $\g$, and
$A^{(i,j)}\in \K^N$, $b^{(i,j)}\in \K$, such that for all $i,j \in [M]$, $\alpha,\beta \in \NN$:
$$f_j(\beta)*f_i(\alpha)=(A^{(i,j)}\cdot \alpha+b^{(i,j)})f_i(\alpha+\beta),$$
where we denote by $\cdot$ the usual inner product of $\K^N$.  The elements $A^{(i,j)}$ and $b^{(i,j)}$ will be called the structure coefficients of $\g$.
\end{defi}

{\bf Example.} We take $M=N$. The pre-Lie product of the $N$-dimensional  Faà di Bruno Lie algebra is given by:
$$f_j(\beta)*f_i(\alpha)=(\alpha_j+\delta_{i,j})f_i(\alpha+\beta).$$
Here, $A^{(i,j)}=\epsilon_j$, and $b^{(i,j)}=\delta_{i,j}$.\\

Let $(\g,*)$ be a deg1 pre-Lie algebra of structure coefficients $A^{(i,j)}$ and $b^{(i,j)}$. 
Let $\lambda_i\in \K-\{0\}$ for all $i\in [M]$. We put $g_i(\alpha)=\lambda_i f_i(\alpha)$ for all $i\in [M]$, $\alpha\in \NN$. Then:
$$g_j(\beta)*g_i(\alpha)=(\lambda_j A^{(i,j)}\cdot \alpha+\lambda_j b^{(i,j)})g_i(\alpha+\beta).$$
So the deg1 pre-Lie algebra with structure coefficients $A^{(i,j)}$ and $b^{(i,j)}$ is isomorphic to the deg1 pre-Lie algebra
with structure coefficients $\lambda_j A^{(i,j)}$ and $\lambda_j b^{(i,j)}$: we shall say that these two pre-Lie algebras are equivalent.
Our aim in this section  is to find all deg1 pre-Lie algebras, up to equivalence.

\begin{lemma}
Let $\g$ be a vector space with a basis $(f_i(\alpha))_{i\in [M], \alpha \in \NN}$, elements $A^{(i,j)}\in \K^N$, $b^{(i,j)}\in \K$, for $i,j \in [M]$.
We define a product $*$ on $\g$ by:
$$f_j(\beta)*f_i(\alpha)=(A^{(i,j)}\cdot \alpha+b^{(i,j)})f_i(\alpha+\beta).$$
Then $(\g,*)$ is a pre-Lie algebra if, and only if, for all $i,j,k\in [M]$:
\begin{align}
\label{C1} &(A^{(i,j)}=0\mbox{ and }b^{(i,j)}=0)\mbox{ or }(A^{(i,j)}=A^{(i,k)}),\\
\label{C2} &A^{(i,j)}b^{(j,k)}=A^{(i,k)}b^{(k,j)},\\
\label{C3} &b^{(i,j)}b^{(j,k)}=b^{(i,k)}b^{(k,j)}.
\end{align}\end{lemma}

\begin{proof}
Let $\alpha,\beta,\gamma\in \NN$, $i,j,k\in [M]$. Then:
\begin{align*}
&(f_k(\gamma)*f_j(\beta))*f_i(\alpha)-f_k(\gamma)*(f_j(\beta)*f_i(\alpha))\\
&=(A^{(i,j)}\cdot \alpha+b^{(i,j)})(A^{(j,k)}\cdot \beta+b^{(j,k)})f_i(\alpha+\beta+\gamma)\\
&-(A^{(i,j)}\cdot \alpha+b^{(i,j)})(A^{(i,k)}\cdot(\alpha+\beta)+b^{(i,k)})f_i(\alpha+\beta+\gamma)\\
&=(A^{(i,j)}\cdot \alpha+b^{(i,j)})((A^{(j,k)}-A^{(i,k)})\cdot\beta-A^{(i,k)}\cdot \alpha+b^{(j,k)}-b^{(i,k)})f_i(\alpha+\beta+\gamma).
\end{align*}
Consequently:
\begin{align*}
&\mbox{$(\g,*)$ is pre-Lie}\\
&\Longleftrightarrow \forall i,j,k \in [M], \forall \alpha \in \NN,\\
&\begin{cases}
(A^{(i,j)}\cdot \alpha+b^{(i,j)})(A^{(j,k)}-A^{(i,k)})=0,\\
(A^{(i,k)}\cdot \alpha+b^{(i,k)})(A^{(k,j)}-A^{(i,j)})=0,\\
(A^{(i,j)}\cdot \alpha+b^{(i,j)})(b^{(j,k)}-b^{(i,k)}-A^{(i,k)}\cdot \alpha)=
(A^{(i,k)}\cdot \alpha+b^{(i,k)})(b^{(k,j)}-b^{(i,j)}-A^{(i,j)}\cdot \alpha),
\end{cases}\\
&\Longleftrightarrow \forall i,j,k \in [M], \\
&\begin{cases}
A^{(i,j)}=0\mbox{ or }A^{(j,k)}=A^{(i,k)},\\
b^{(i,j)}=0\mbox{ or }A^{(j,k)}=A^{(i,k)},\\
A^{(i,j)}(b^{(j,k)}-b^{(i,k)})-b^{(i,j)}A^{(i,k)}=A^{(i,k)}(b^{(k,j)}-b^{(i,j)})-b^{(i,k)}A^{(i,j)},\\
b^{(i,j)})(b^{(j,k)}-b^{(i,k)})=b^{(i,k)}(b^{(k,j)}-b^{(i,j)}),
\end{cases}
\end{align*}
which is equivalent to conditions (\ref{C1})-(\ref{C3}). \end{proof}

\begin{prop}
Let $[M]=I_0\sqcup \ldots \sqcup I_k$ be a partition of $[M]$, such that $I_1,\ldots,I_k\neq \emptyset$ (note that $I_0$ may be empty), 
$A_1,\ldots,A_k\in \K^N$, $b_1,\ldots,b_p \in \K$, and $b^{(i)}_p \in \K$ for all $i\in I_0$ and $p\in [k]$. We define a deg1 pre-Lie algebra by:
\begin{align*}
A^{(i,j)}&=\begin{cases}
A_q\mbox{ if }j\in I_q, q\geq 1,\\
0\mbox{ if } j\in I_0.
\end{cases}&
b^{(i,j)}&=\begin{cases}
\delta_{p,q} b_q\mbox{ if }j\in I_q, q\geq 1, i\in I_p, p\geq 1,\\
0\mbox{ if }j\in I_0,\\
b^{(i)}_q\mbox{ if }j\in I_q, q\geq 1, i\in I_0.
\end{cases}\end{align*}
This pre-Lie algebra will be called the fundamental deg1 pre-Lie algebra of parameters $I=(I_0,\ldots,I_k)$, $A=(A_1,\ldots,A_k) \in M_{N,k}(\K)$,
$b=(b_1,\ldots,b_k) \in \K^k$ and $b^{(i,j)}$.
 \end{prop}

\begin{proof} Direct verifications prove that these structure coefficients satisfy conditions (\ref{C1})-(\ref{C3}). \end{proof}\\

{\bf Remarks.} \begin{enumerate}
\item For example, the Faà di Bruno pre-Lie algebra of dimension $N$ is fundamental, with $I_j=\{j\}$ for all $j\in [M]$, $I_0=\emptyset$,
$A=I_N$ and $b=(1,\ldots,1)$.
\item The pre-Lie product of such a pre-Lie algebra is given in the following way: if $i\in I_p$, $j\in I_q$, $\alpha,\beta \in \NN$,
\begin{align*}
f_j(\beta)*f_i(\alpha)&= \begin{cases}
(A_q\cdot \alpha+\delta_{p,q} b_q)f_i(\alpha+\beta)\mbox{ if }p,q\neq 0,\\
(A_q\cdot \alpha+b^{(i)}_q)f_i(\alpha+\beta)\mbox{ if }p=0,q\neq 0,\\
0\mbox{ if }q=0.
\end{cases}\end{align*}
\end{enumerate}

\subsection{Classification of deg1 pre-Lie algebras}

Let $\g$ be a deg1 pre-Lie algebra. We attach to it an oriented graph $G(\g)$, defined as follows:
\begin{itemize}
\item The vertices of $G(\g)$ are the elements of $[M]$.
\item There exists an oriented edge from $i$ to $j$ if, and only if, $b^{(i,j)}\neq 0$.
\end{itemize} 

We shall write $i\longrightarrow j$ if there is an oriented edge from $i$ to $j$ in $G(\g)$.

\begin{lemma} \label{lemmegraphe}
Let $\g$ be a fundamental deg1 pre-Lie algebra and let $i\longrightarrow j\longrightarrow k$ in $G(\g)$. Then, in $G(\g)$:
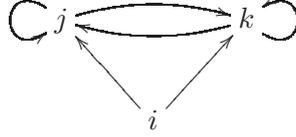
$$\xymatrix{j \ar@(ul,dl)[] \ar@/^.5pc/[rr]&&k \ar@/^.5pc/[ll]\ar@(dr,ur)[] \\
&i \ar[ru] \ar[lu]&}$$
\end{lemma}

\begin{proof}
By condition (\ref{C3}), if $i\longrightarrow j\longrightarrow k$, then $b^{(i,j)}b^{(j,k)}=b^{(i,k)}b^{(k,j)}\neq 0$, so $i\longrightarrow k$ and $k\longrightarrow j$.
With the same argument, as $k\longrightarrow j\longrightarrow k$, $k\longrightarrow k$. As $j\longrightarrow k\longrightarrow j$, $j\longrightarrow j$. 
\end{proof}

\begin{prop}\label{propgraphe}
Let $\g$ be a fundamental deg1 pre-Lie algebra. The graph $G(\g)$ has the following structure:
\begin{enumerate}
\item The set of vertices $[M]$ admits a partition $[M]=I_0\sqcup \ldots \sqcup I_k$.
\item For all $1\leq p\leq k$, the complete subgraph of $G(\g)$ whose vertices are the elements of $I_p$ is, either complete, either an isolated vertex.
\item For all $i\in I_0$, there exists $D(i)\subseteq [k]$, such that for all $j \in [M]$, $i\longrightarrow j$ if, and only if,
$\displaystyle j\in \bigsqcup_{p\in D(i)} I_p$.
\item If $i\in I_0$, there is no vertex $j$ such that $j\longrightarrow i$.
\end{enumerate} \end{prop}

\begin{proof} {\it First step.} Let $i_0\in [M]$. For all $p\geq 1$, we denote by $J_p$ the sets of vertices $j\in [M]$, such that there exists
$i_1,\ldots,i_{p-1}\in [M]$, $i_0\longrightarrow i_1\longrightarrow \ldots \longrightarrow i_{p-1} \longrightarrow j$.
We put $J=\displaystyle \bigcup_{p\geq 1} J_p$ and we consider a connected component $K$ of the  subgraph of $G(\g)$ of vertices $J$.
Let us prove that $K$ is either complete, or is an isolated vertex. 
First, observe that if $j\longrightarrow k$ in $K$, by definition of $J$, there exists $j_{p-1}$, such that $j_{p-1}\longrightarrow j\longrightarrow k$. 
By lemma \ref{lemmegraphe}, $\{j,k\}$ is a complete subgraph of $K$.

If $K$ has no edge, as it is connected, it is an isolated vertex; let us assume it has at least one edge $j\longrightarrow k$. By the preceding observation,
$\{j,k\}$ is a complete subgraph of $K$, so $K$ contains complete subgraphs. Let $L$ be a maximal complete subgraph of $K$. 
If $L\subsetneq K$, as $K$ is connected, there exists $k\in K\setminus L$, $l\in L$, such that $k\longrightarrow l$ or $l\longrightarrow k$. 
We already observed that $\{k,l\}$ is complete in both cases. Let $l'\in L$. As $L$ is complete, then $k\longrightarrow l\longrightarrow l'$ 
and $l'\longrightarrow l \longrightarrow k$: by lemma \ref{lemmegraphe}, $k\longrightarrow l'$, and $l'\longrightarrow k$: $L\sqcup\{k\}$ is complete,
which contradicts the maximality of $L$. So $K=L$ is complete.\\

{\it Second step.} We denote by $I_0$ the set of vertices $i$ such that there is no $j$ with $j\longrightarrow i$. 
Let $K$ be a connected component of the subgraph of vertices $[M]\setminus I_0$. If $k\in K$, then $k\notin I_0$,
so there exists $j\in I$, such that $j\longrightarrow k$. By the first step, $K$ is an isolated vertex or is complete. We denote by $I_1\sqcup \ldots \sqcup I_k$
the decomposition of $[M]\setminus I_0$ in connected components. Let $i_0 \in I_0$, and $j$ such that $i_0\longrightarrow j$.
Then $j\notin I_0$, so there exists $p\geq 1$, $j\in I_p$. If $I_p$ is an isolated vertex, then $i_0\longrightarrow j'$ for any $j'\in I_p$.
If $I_p$ is complete, for any $j'\in I_p$, then $i_0\longrightarrow j\longrightarrow j'$, so $i_0\longrightarrow j'$ by lemma \ref{lemmegraphe}.
Denoting by $D(i_0)$ the set of $p$ such that there exists $j\in I_p$ with $i_0\longrightarrow j$, then $i_0\longrightarrow j$
if, and only if, $j\in I_p$ for a $p\in D(i_0)$. \end{proof}

\begin{theo} \label{theoclassif}
Let $\g$ be a deg1 pre-Lie algebra. Up to an equivalence, it is the direct sum of fundamental deg1 pre-Lie algebras.
\end{theo}

\begin{proof}  {\it First case.} We assume first that $G(\g)$ is complete. Let us choose $i_0 \in I$. For all $j$, $b^{(i_0,j)}\neq 0$: up to an equivalence,
we assume that $b^{(i_0,j)}=1$ for all $j$. 
Condition (\ref{C3}), with $i=i_0$ becomes: for all $j,k$, $b^{(j,k)}=b^{(k,j)}$. Still by condition (\ref{C3}), as $b^{(j,k)}=b^{(k,j)}\neq 0$,
for all $i,j,k$, $b^{(i,j)}=b^{(i,k)}$. Hence, for all $i,j$:
$$b^{(i,j)}=b^{(i,i_0)}=b^{(i_0,i)}=1.$$
Condition (\ref{C1}) becomes: for all $i,j,k$, $A^{(j,k)}=A^{(i,k)}$. We denote by $A^{(k)}$ the unique vector such that
$A^{(i,k)}=A^{(k)}$ for all $i$.  Condition (\ref{C2}) becomes: for all $j,k$, $A^{(k)}=A^{(j)}$. So there exists a unique vector $A$,
such that for all $i,j$, $A^{(i,j)}=A$. Finally, $\g$ is a fundamental deg1 pre-Lie algebra, with $[M]=I_1$.\\

{\it Second case.} We assume that $G(\g)$ is connected. We use the notations of proposition \ref{propgraphe}.
If there is an edge from $i$ to $j$, by condition (\ref{C1}), for all $k$, $A^{(j,k)}=A^{(i,k)}$. By connectivity, there exists vectors $A^{(k)}$,
such that for all $i,j,k$, $A^{(i,k)}=A^{(j,k)}=A^{(k)}$. We consider the pre-Lie subalgebra $\g_p$ of $\g$ generated by the elements 
$f_i(\alpha)$, $i\in I_p$, $\alpha \in \NN$. They are deg1 pre-Lie algebras; if $p\geq 1$ and $I_p$ is not a single element, 
then the graph associated to $\g_p$ is complete. By the first step, up to an equivalence, we can assume that $A^{(k)}$ is constant on $I_p$:
there exists a vector $A_p$ such that $A^{(k)}=A_p$ for all $k\in I_p$, $p\geq 1$. Moreover, there exists a scalar $b_p$,
such that $b^{(i,j)}=b_p$ for all $i,j \in I_p$, if $p\geq 1$. 

Let $j\in I_0$. By connectivity of $G(\g)$, and by definition of $I_0$, there exists $k$ such that $j\longrightarrow k$, so $b^{(j,k)}\neq 0$
and $b^{(k,j)}=0$. By condition (\ref{C2}), $A^{(i,j)}=0$ for all $i$, so $A^{(j)}=0$ if   $j\in I_0$. 

By definition of the graph, if $i\in I_p$, $j\in I_q$, $p,q \geq 1$ and $p\neq q$, then $b^{(i,j)}=0$.
If $j\in I_0$, then $b^{(i,j)}=0$ for all $i$. Let $i\in I_0$, $j,k\in I_p$, $p\geq 1$. If $j=k$, then $b^{(i,j)}=b^{(i,k)}$.
If $j\neq k$, then $I_p$ is complete and $j\longrightarrow k$ in $G(\g)$: $b^{(j,k)}=b^{(j,k)}\neq 0$. By condition (\ref{C3}),
$b^{(i,k)}=b^{(i,j)}$. So there exists $b^{(i)}_p$, such that $b^{(i,j)}=b^{(i)}_p$ for all $j\in I_p$. Finally, the structure coefficients are given in the following arrays:
\begin{align*}
A^{(i,j)}&: \begin{array}{|c|c|c|c|c|}
\hline i\setminus j&I_0&I_1&\ldots&I_k\\
\hline I_0&0&A_1&\ldots&A_k\\
\hline I_1&0&\vdots&&\vdots\\
\hline \vdots&0&\vdots&&\vdots\\
\hline I_k&0&A_1&\ldots&A_k\\
\hline \end{array}&
b^{(i,j)}&:  \begin{array}{|c|c|c|c|c|}
\hline i\setminus j&I_0&I_1&\ldots&I_k\\
\hline I_0&0&b^{(i)}_1&\ldots&b^{(i)}_k\\
\hline I_1&0&b_1&\ldots&0\\
\hline \vdots&0&\vdots&\ddots&\vdots\\
\hline I_k&0&0&\ldots&b_k\\
\hline \end{array}\end{align*}
So this is a fundamental deg1 pre-Lie algebra.\\

{\it General case}. Let $G_1,\ldots,G_l$ be the connected components of $G(\g)$. By the second step, up to an equivalence of $\g$,
the pre-Lie subalgebra of $\g$ corresponding to these subgraphs are fundamental deg1 pre-Lie algebras. 

{\it First subcase}. Let us assume that there exists $i\in G_p$, $j\in G_q$, with $p\neq q$, such that $A^{(i,j)}\neq 0$. 
By condition (\ref{C1}), for all $k$, $A^{(j,k)}=A^{(i,k)}$. By connectivity of $G_p$ and $G_q$, we deduce that for all $i'\in G_p$,
$j'\in G_q$, for all $k$, $A^{(i',k)}=A^{(j',k)}$. 

{\it Second subcase}. Let us assume that for all $i\in G_p$, $j\in G_q$, $A^{(i,j)}=0$. As $b^{(i,j)}=0$,
for all $\alpha,\beta \in \NN$, for all $i \in G_p$, $j\in G_q$, $f_j(\beta)*f_i(\alpha)=0$.\\

We define an equivalence relation $\sim$ on $[M]$ in the following way:
$i\sim j$ if for all $k$, $A^{(i,k)}=A^{(j,k)}$. The first subcase implies that the equivalence classes are disjoint union 
of $G_p$: we denote them by $H_1,\ldots,H_n$. The second step gives that the corresponding subalgebras $\g_1,\ldots \g_n$
are fundamental deg1 pre-Lie algebras. By the second subcase, $\g=\g_1\oplus\ldots \oplus \g_n$. \end{proof}

\subsection{SDSE associated to a deg1 pre-Lie algebra}

We here describe the dual of the enveloping algebra of a deg1 pre-Lie algebra, as a subalgebra of a Hopf algebra of decorated rooted trees.
We use for this the Guin-Oudom extension of the pre-Lie product \cite{Oudom}.

\begin{lemma}
Let $\g$ be a fundamental deg1 pre-Lie algebra. For all $i\in I_p$, $p\neq 0$, $\alpha,\beta_1,\ldots,\beta_k \in \NN$:
$$f_{j_1}(\beta_1)\ldots f_{j_k}(\beta_k)*f_i(\alpha)=\begin{cases}
0\mbox{ if one of the $j_p$ is in $I_0$},\\
\displaystyle \prod_{q=1}^k \prod_{r=0}^{\varepsilon_q-1} (A_q\cdot \alpha+b_q(\delta_{p,q}-r)) f_i(\alpha+\beta_1+\ldots+\beta_k)\mbox{ otherwise},
\end{cases}$$
where $\varepsilon_q=\sharp\{p\in [k]\mid j_p\in I_q\}$.
If $i\in I_0$:
$$f_{j_1}(\beta_1)\ldots f_{j_k}(\beta_k)*f_i(\alpha)=\begin{cases}
0\mbox{ if one of the $j_p$ is in $I_0$},\\
\displaystyle \prod_{q=1}^k \prod_{r=0}^{\varepsilon_q-1} (A_q\cdot \alpha+b^{(i)}_q-rb_q) f_i(\alpha+\beta_1+\ldots+\beta_k)\mbox{ otherwise}.
\end{cases}$$
\end{lemma}

\begin{proof}
We proceed by induction on $k$. The result is obvious if $k=1$. Let us assume the result at rank $k$. We assume that $i\in I_p$, $p\geq 1$.
We put:
$$f_{j_1}(\beta_1)\ldots f_{j_k}(\beta_k)*f_i(\alpha)=P_{(j_1,\ldots,j_k)}(\alpha)f_i(\alpha+\beta_1+\ldots+\beta_k).$$
Then:
\begin{align*}
f_{j_1}(\beta_1)\ldots f_{j_{k+1}}(\beta_{k+1})*f_i(\alpha)&=f_{j_{k+1}}(\beta_{k+1})*\left(f_{j_1}(\beta_1)\ldots f_{j_k}(\beta_k)*f_i(\alpha)\right)\\
&-\sum_{p=1}^k f_{j_1}(\beta_1)\ldots (f_{j_{k+1}}(\beta_{k+1})*f_{j_p}(\beta_p))\ldots f_{j_k}(\beta_k)*f_i(\alpha).
\end{align*}
If $j_{k+1}\in I_0$, this is zero. Let us assume that $j_{k+1}\in I_q$, $q\geq 1$. For all $p$, let $bl(p)$ be the unique $r$ such that $j_p \in I_r$. Then:
\begin{align*}
&f_{j_1}(\beta_1)\ldots f_{j_{k+1}}(\beta_{k+1})*f_i(\alpha)\\
&=P_{(j_1,\ldots,j_k)}(\alpha)f_{j_{k+1}}(\beta_{k+1})*f_i(\alpha+\beta_1+\ldots+\beta_k)\\
&-\sum_{p=1}^k (A_q\cdot \beta_p+\delta_{bl(p),q} b_q) f_{j_1}(\beta_1)\ldots f_{j_p}(\beta_p+\beta_{k_{k+1}})\ldots f_{j_k}(\beta_k)*f_i(\alpha)\\
&=P_{(j_1,\ldots,j_k)}(\alpha)(A_q\cdot(\alpha+\beta_1+\ldots+\beta_k)+\delta_{p,q}b_q)f_i(\alpha+\beta_1+\ldots+\beta_{k+1})\\
&-P_{(j_1,\ldots,j_k)}(\alpha)(A_q\cdot(\beta_1+\ldots+\beta_k)+b_q(\varepsilon_q-1))f_i(\alpha+\beta_1+\ldots+\beta_{k+1})\\
&=P_{(j_1,\ldots,j_k)}(\alpha)(A_q\cdot \alpha+b_q(\delta_{p,q}-\varepsilon_q+1))f_i(\alpha+\beta_1+\ldots+\beta_{k+1}).
\end{align*}
The computation is similar if $i\in I_0$. \end{proof}\\

We shall write shortly:
$$f_{j_1}(\beta_1)\ldots f_{j_k}(\beta_k)*f_i(\alpha)=P_{(j_1,\ldots,j_k)}(\alpha)f_i(\alpha+\beta_1+\ldots+\beta_k).$$

Let $D_{M,N}=[M]\times \N^N_*$. Recall that $\g^{D_{M,N}}$ is the free pre-Lie algebra generated by the rooted trees $\tdun{$d$}$, $d\in D_{M,N}$.
The set $D_{M,N}$ is $\N^N$-graded, with $deg(i,\alpha)=\alpha$, and this graduation is connected.

If $\g$ is a deg1 pre-Lie algebra, one defines a connected $\N^N$-graduation of the pre-Lie $\gN$ by putting $f_i(\alpha)$ homogeneous of degree $\alpha$.
 We define a pre-Lie algebra morphism:
$$\phi:\left\{\begin{array}{rcl}
\gN&\longrightarrow&\g\\
\tdun{$(i,\alpha) $}\hspace{.2cm}&\longrightarrow&f_i(\alpha).
\end{array}\right.$$

\begin{lemma}
Let $T\in \T^{D_{M,N}}$. We denote by $(r(T),d(T))$ the decoration of the root of $T$. 
There exists a scalar $\lambda_T$, such that:
$$\phi(T)=\lambda_T f_{r(T)}(deg(T)).$$
These coefficients can be inductively defined by:
$$\lambda_T=\begin{cases}
1\mbox{ if }T=\tdun{$(i,\alpha)$}\hspace{.45cm},\\
\lambda_{T_1}\ldots\lambda_{T_k} P_{(r(T_1),\ldots,r(T_k))}(\alpha)\mbox{ if }t=B_{(i,\alpha)}(T_1\ldots T_k).
\end{cases}$$ 
\end{lemma}

\begin{proof} We proceed by induction on the number $n$ of vertices of $T$. It is obvious if $n=1$. Let us assume the result at all rank $<n$, $n\geq 2$.
We put $t=B_{(i,\alpha)}(T_1\ldots T_k)$. Then $T=T_1\ldots T_k *\tdun{$(i,\alpha) $}\hspace{5mm}$, so:
\begin{align*}
\phi(T)&=\phi(T_1)\ldots \phi(T_k) f_i(\alpha)\\
&=\lambda_{T_1}\ldots \lambda_{T_k} f_{r(T_1)}(|T_1|)\ldots f_{r(T_k)}(|T_k|)*f_i(\alpha)\\
&=\lambda_{T_1}\ldots \lambda_{T_k} P_{(r(T_1),\ldots,r(T_k))}(\alpha) f_i(\alpha+deg(T_1)+\ldots+deg(T_k))\\
&=\lambda_{T_1}\ldots \lambda_{T_k} P_{(r(T_1),\ldots,r(T_k))}(\alpha) f_i(deg(t)).
\end{align*}
Hence, the result holds for all $n$. \end{proof}\\

By duality, we obtain a Hopf algebra morphism:
$$\phi^*:\left\{\begin{array}{rcl}
\mathcal{U}(\g)^*&\longrightarrow&\hN\\
f_i(\alpha)^*&\longrightarrow&\displaystyle \sum_{deg(T)=\alpha, r(T)=i} \frac{\lambda_T}{s_T}T.
\end{array}\right.$$
We put $\mu_T=\frac{\lambda_T}{s_T}$ for any rooted tree $T\in \T^{D_{M,N}}$, and, for any $i\in [M]$:
$$X_i=\sum_{r(T)=i} \mu_TT.$$
If $t=B_{(i,\alpha)}\left(T_1^{\beta_1}\ldots T_l^{\beta_l}\right)$, where $T_1,\ldots,T_k$ are distinct trees, with $i\in I_p$, $p\geq 1$, 
denoting by $\varepsilon_q$ the number of trees $t'$ in $T_1^{\beta_1}\ldots T_l^{\beta_l}$ such that $r(T_i)\in I_q$:
\begin{align*}
\mu_T&=\frac{\lambda_{T_1}^{\beta_1}\ldots \lambda_{T_l}^{\beta_l}}
{s_{T_1}^{\beta_1}\ldots s_{T_l}^{\beta_l} \beta_1!\ldots \beta_l!}\prod_{q=1}^k \prod_{r=0}^{\varepsilon_q-1}
(A_q\cdot \alpha+b_q(\delta_{p,q}-r))\\
&=\mu_{T_1}^{\alpha_1}\ldots\mu_{T_l}^{\alpha_l} \frac{\varepsilon_1!\ldots \varepsilon_k!}{\beta_1!\ldots \beta_l!}
\prod_{q=1}^k \frac{1}{\varepsilon_q!}\prod_{r=0}^{\varepsilon_q-1}(A_q\cdot \alpha+b_q(\delta_{p,q}-r)).
\end{align*}
Consequently, if $i\in I_p$, $p\geq 1$:
$$X_i=\sum_{\alpha \in \NN} B_{(i,\alpha)}\left(
\prod_{q=1}^k f_{i,q,\alpha}\left(\sum_{j\in I_q}X_j\right)\right),$$
with:
$$f_{i,q,\alpha}(X)=F_{A_q\cdot \alpha+b_q\delta_{p,q},b_q}(X)=F_{A_q\cdot \alpha,b_q}(X) F_{b_q\delta_{p,q},b_q}(X)=
F_{A_q\cdot \alpha,b_q}(X) (1+b_qX)^{\delta_{p,q}}.$$
A similar computation for $i\in I_0$ gives:
$$X_i=\sum_{\alpha \in \NN} B_{(i,\alpha)}\left(
\prod_{q=1}^kF_{A_q\cdot \alpha,b_q}\left(\sum_{j\in I_q}X_j\right)
\prod_{q=1}^k F_{b^{(i)}_q,b_q}\left(\sum_{j\in I_q}X_j\right)\right).$$

We proved:

\begin{theo}\label{theo14}
Let $[M]=I_0\sqcup I_1 \sqcup \ldots \sqcup I_k$, $A_1,\ldots,A_k \in \K^N$, $b_1,\ldots,b_k \in \K$, $b^{(i)}_1,\ldots,b^{(i)}_k\in \K$ for all $i\in I_0$.
We consider the following SDSE:
\begin{align*}
\forall i\in I_p,\: p\geq 1,\: X_i&=\sum_{\alpha \in \NN}B_{(i,\alpha)}\left(
\prod_{q=1}^kF_{A_q\cdot \alpha,b_q}\left(\sum_{j\in I_q}X_j\right)\left(1+b_p\sum_{j\in I_p} X_j\right)\right),\\
\forall i\in I_0,\:X_i&=\sum_{\alpha \in \NN}B_{(i,\alpha)}\left(\prod_{q=1}^kF_{A_q\cdot \alpha,b_q}\left(\sum_{j\in I_q}X_j\right)
\prod_{q=1}^k F_{b^{(i)}_q,b_q}\left(\sum_{j\in I_q}X_j\right)\right).
\end{align*}
The $\N^N$-graded subalgebra of $\hN$ generated by the unique solution of this SDSE is Hopf. Its dual is the enveloping algebra of
the fundamental deg1 pre-Lie algebra associated to $I$, $A$ and $b$. 
\end{theo}

{\bf Example.} We choose $N=M$, $I=\{1\}\sqcup \ldots \sqcup \{N\}$, $A=I_N$ and $b_i=1$ for all $i\in [N]$. The associated Hopf SDSE is:
\begin{align*}
(S):\: \forall i\in [N], \: X_i=\sum_{\alpha \in \NN}B_{(i,\alpha)}\left(
\prod_{q=1}^N (1+X_q)^{\alpha_q}(1+X_i)\right).
\end{align*}
This is related to the SDSE described in (\ref{system}). We only conserve as decorations the elements of:
$$D'=\{deg(\alpha) \mid \alpha \in D\}.$$
For all $\alpha=(\alpha_0,\ldots,\alpha_k) \in [N]^{k+1}$, we put $B'_{\alpha}=B_{(\alpha_0,deg(\alpha))}$. The SDSE $(S)$ becomes:
\begin{align*}
(S'):\: \forall i\in [N], \: X_i&=\sum_{\alpha \in\N^k} B'_{(i,\alpha)}\left(
\prod_{q=1}^N (1+X_q)^{\sum_{p, \alpha_p=q}d_p}(1+X_i)^{d_0+1}\right)\\
\Longleftrightarrow \forall i\in [N], \:X_i&=\sum_{\alpha \in\N^k} B'_{(i,\alpha)}\left((1+X_{\alpha_1})^{d_1}\ldots (1+X_{\alpha_k})^{d_k}(1+X_i)^{d_0+1}\right).
\end{align*}
This is the system of (\ref{system}), which is consequently a Hopf SDSE.

\section{Group associated to a fundamental pre-Lie algebra}

\subsection{Lie algebra associated to a fundamental pre-Lie algebra}

\begin{prop}
Let $\g$ be a fundamental deg1 pre-Lie algebra, with parameters $I$, $A$ and $b$.
We denote by $r$ the rank of $A$. Then $\g$ is isomorphic, as a Lie algebra, to a fundamental deg1 pre-Lie algebra $\g'$ with structure coefficients 
given by:
\begin{align*}
A'^{(i,j)}&: \begin{array}{|c|c|c|}
\hline i\setminus j&1\ldots r&r+1\ldots M\\
\hline1\ldots M&A'_j&0\\
\hline \end{array}&
b'^{(i,j)}&:  \begin{array}{|c|c|c|}
\hline i\setminus j&1\ldots k&k+1\ldots M\\
\hline 1\ldots k&0&0\\
\hline k+1\ldots M&b'^{(i)}_j&0\\
\hline \end{array}\\
A'&=\left(\begin{array}{c}
I_r\\ *
\end{array}\right), \end{align*}
with $0\leq r \leq k\leq M$. We shall say that such a fundamental deg1 pre-Lie algebra is reduced.
\end{prop}

\begin{proof}  {\it First step}. 
For any $p\geq 1$, let us fix $i_0 \in I_p$. If $i\in I_p\setminus\{i_0\}$, we put $g_i(\alpha)=f_i(\alpha)-f_{i_0}(\alpha)$ for all $\alpha \in \NN$.
If $j \in I_q$, $q\neq 0$:
$$f_j(\beta)*g_i(\alpha)=(A_q\cdot \alpha+b_q \delta_{p,q})g_i(\alpha+\beta).$$
Consequently: 
\begin{align*}
g_j(\beta)*g_i(\alpha)&=0 \mbox{ if }j\in I_p\setminus\{i_0\},&
f_j(\beta)*g_i(\alpha)&=0\mbox{ if }j\in I_0.
\end{align*}
Replacing the elements $f_i(\alpha)$ by $g_i(\alpha)$ for all $i\in I_p\setminus\{i_0\}$, these computations proves that $\g$ is isomorphic to
a deg1 pre-Lie algebra $\g'$, with $[M]=I'_0\sqcup \ldots \sqcup I'_k$, such that 
$$I'_q=\begin{cases}
\{i_0\}\mbox{ if }q=p,\\
I_0\sqcup I_p\setminus\{i_0\}\mbox{ if }q=0,\\
I_q\mbox{ otherwise.}
\end{cases}$$
Proceding in this way for all $p$, and after a reindexation, we obtain that $\g$ is isomorphic to a fundamental deg1 pre-Lie algebra with:
\begin{align*}
A^{(i,j)}&: \begin{array}{|c|c|c|c|c|}
\hline i\setminus j&1&\ldots&k&k+1\ldots M\\
\hline1\ldots M&A_1&\ldots&A_k&0\\
\hline \end{array}\\ \\
b^{(i,j)}&:  \begin{array}{|c|c|c|c|c|c|}
\hline i\setminus j&1&\ldots&\ldots& k&k+1\ldots M\\
\hline 1&b_1&0&\ldots&0&0\\
\hline \vdots&0&\ddots&\ddots&\vdots&\vdots\\
\hline \vdots&\vdots&\ddots&\ddots&0&\vdots\\
\hline k&0&\ldots&0&b_k&\vdots\\
\hline k+1\ldots M&b^{(i)}_1&\ldots&\ldots&b^{(i)}_k&0\\
\hline \end{array}\end{align*}
If $1\leq i,j \leq k$, in $\g'$:
\begin{align*}
[f_j(\beta),f_i(\alpha)]&=(A_j\cdot \alpha+\delta_{i,j}b_i)f_i(\alpha+\beta)-(A_i\cdot \beta+\delta_{i,j}b_j)f_j(\alpha+\beta)\\
&=A_j\cdot \alpha f_i(\alpha+\beta)-A_i\cdot \beta f_j(\alpha+\beta).
\end{align*}
Hence, the Lie bracket of $\g$ does not depend of $b$. \\

{\it Second step.} Up to a Lie algebra isomorphism, we can now assume that $b_1=\ldots=b_k=0$. Let $P \in GL_k(\K)$. For all $i\in [k]$, we put:
$$g_i(\alpha)=\sum_j p_{j,i} f_j(\alpha).$$
Then $(g_i(\alpha))_{i\leq k,\alpha \in \NN}\sqcup (f_i(\alpha))_{i>k, \alpha \in \NN}$ is a basis of $\g$. Moreover, if $i,j \in [k]$:
\begin{align*}
g_j(\beta)*g_j(\alpha)&=\sum_{i',j'}p_{j',j}p_{i',i} f_{j'}(\beta)*f_{i'}(\alpha)\\
&=\sum_{i',j'} p_{j',j}p_{i',i}A_{j'}\cdot \alpha f_{i'}(\alpha+\beta)\\
&=\left(\sum_{j'} p_{j',j}A_{j'}\right)\cdot \alpha g_i(\alpha+\beta).
\end{align*} 
Similar computations give, if $1\leq j\leq k<i\leq N$:
$$g_j(\beta)*f_i(\alpha)=\left(\left(\sum_{j'} p_{j',j}A_{j'}\right)\cdot \alpha+\left(\sum_{j'}p_{j',j}b^{(i)}_{j'}\right)\right)f_i(\alpha+\beta).$$
Moreover, if $1\leq i\leq k<j \leq N$:
$$f_j(\beta)*g_i(\alpha)=0.$$
Hence, $\g$ is isomorphic, as a Lie algebra, to the fundamental deg1 pre-Lie $\g'$, with $A'=A P$, and $b'^{(i,j)}=0$ if $i,j \leq k$. 
Up to a permutations of the rows and the columns of $A$, we can assume that:
$$A=\left(\begin{array}{cc}
A_1&A_2\\
A_3&A_4
\end{array}\right),$$
with $A_1\in GL_r(\K)$. As $r=Rank(A)$, there exists $Q\in M_{r,k-r}$, such that:
$$\left(\begin{array}{c}A_2\\A_4\end{array}\right)=\left(\begin{array}{c}A_1\\A_3\end{array}\right)Q.$$
We then take:
$$P=\left(\begin{array}{cc}
A_1^{-1}&-Q\\ 0&I_{k-r}
\end{array}\right),$$
and then:
$$A'=\left(\begin{array}{cc}
I_r&0\\
*&0
\end{array}\right),$$
which finally gives the announced result. \end{proof}

\subsection{Group associated to a reduced deg1 pre-Lie algebra}

{\bf Notations.} Let $p \in \N^*$ and $q\in \N$. We fix a matrix $\calB\in M_{q,p}(\K)$. For all $i\in [p]$, we denote:
$$\bfG_i=\{x_i(1+F)\mid F\in \K[[x_1,\ldots,x_p,y_1,\ldots,y_q]]_+\}\subseteq\K[[x_1,\ldots,x_p,y_1,\ldots,y_q]]_+.$$

\begin{prop}
Let $\bfG_\calB=\bfG_1\times \ldots \times \bfG_p\subseteq \K[[x_1,\ldots,x_p,y_1,\ldots,y_q]]^p$, with the product defined in the following way: 
if $F=(F_1,\ldots,F_p)$ and $G=(G_1,\ldots,G_p)\in \bfG_\calB$,
$$F\bullet G=G\left(F_1,\ldots,F_p,y_1\left(\frac{F_1}{x_1}\right)^{\calB_{1,1}}\hspace{-4mm}\ldots \left(\frac{F_p}{x_p}\right)^{\calB_{1,p}},\ldots,
y_q\left(\frac{F_1}{x_1}\right)^{\calB_{q,1}}\hspace{-4mm}\ldots \left(\frac{F_p}{x_p}\right)^{\calB_{q,p}}\right).$$
Then $\bfG_\calB$ is isomorphic to the group of characters of a $\N^{p+q}$-graded Hopf algebra $\h_\calB$.
The graded dual of $\h_\calB$ is the enveloping algebra of the reduced deg1 pre-Lie algebra $\g_\calB$ associated to the structure coefficients:
\begin{align*}
A^{(i,j)}&: \begin{array}{|c|c|}
\hline i\setminus j&1\ldots p\\
\hline1\ldots p&A_j\\
\hline \end{array}&
A&=\left( \begin{array}{c}
I_p\\\calB
\end{array}\right)&
b^{(i,j)}&:  \begin{array}{|c|c|}
\hline i\setminus j&1\ldots p\\
\hline1\ldots p&0\\
\hline \end{array}\end{align*}\end{prop}

\begin{proof} We shall write shortly $F\bullet G=G\left(F,Y\left(\frac{F}{x}\right)^\calB\right)$. Let us first prove that $\bfG_\calB$ is a monoid.
Let $F,G,H \in \bfG_\calB$.
\begin{align*}
F\bullet(G\bullet H)&=G\bullet H\left(F,y\left(\frac{F}{x}\right)^\calB\right)\\
&=H\left(G\left(F,y\left(\frac{F}{x}\right)^\calB\right),y\left(\frac{F}{x}\right)^\calB\left(\frac{G\left(F,y\left(\frac{F}{x}\right)^\calB\right)}{F}\right)^\calB\right)\\
&=H\left(G\left(F,y\left(\frac{F}{x}\right)^\calB\right),y\left(\frac{G\left(F,y\left(\frac{F}{x}\right)^\calB\right)}{x}\right)^\calB\right)\\
&=H\left(F\bullet G,y\left(\frac{F\bullet G}{x}\right)^\calB\right)\\
&=(F\bullet G)\bullet H.
\end{align*}
The identity of this monoid is the element $I=(x_1,\ldots,x_p)$.\\

Let $\lambda=(\lambda_1,\ldots,\lambda_p,\mu_1,\ldots,\mu_q) \in (\K^*)^{p+q}$. We define:
$$\phi_\lambda:\left\{\begin{array}{rcl}
\bfG_\calB&\longrightarrow&\bfG_\calB\\
F&\longrightarrow&\left(\frac{1}{\lambda_i}F_i(\lambda_1x_1,\ldots,\lambda_px_p,\mu_1y_1,\ldots,\mu_qy_q)\right)_{1\leq i\leq p}.
\end{array}\right.$$
Let us prove that this defines a action of the torus $T=(\K^*)^{p+q}$ on the monoid $\bfG_\calB$ by automorphisms.
We shall write shortly $\phi_{\lambda,\mu}(F)=\frac{1}{\lambda}F(\lambda x,\mu y)$. Clearly, $\phi_{\lambda,\mu}\circ \phi_{\lambda',\mu'}
=\phi_{\lambda\lambda',\mu\mu'}$ and $\phi_{1,1}=Id_{\bfG_\calB}$, so this is indeed an action. Let $F,G \in \bfG_\calB$.
\begin{align*}
\phi_{(\lambda,\mu)}(F\bullet G)&=\frac{1}{\lambda}G\left(F(\lambda x,\mu y),\mu y\left(\frac{F(\lambda x,\mu y)}{\lambda x}\right)^\calB\right)\\
&=
\phi_{(\lambda,\mu)}(F)\bullet \phi_{(\lambda,\mu)}(G).
\end{align*}

For all $i\in [p]$, $\lambda \in \N_*^{p+q}$, we put:
$$X_i(\lambda):\left\{\begin{array}{rcl}
\bfG_\calB&\longrightarrow&\K\\
G&\longrightarrow&\mbox{coefficient of $x_ix_1^{\lambda_1}\ldots x_p^{\lambda_p}y_1^{\mu_1}\ldots y_q^{\mu_q}$ in $\bfG_i$}.
\end{array}\right.$$
We obtain an action on the torus $T$ on these functions by transposition: 
\begin{align*}
\phi^*_\lambda(X_i(\alpha))(G)&=X_i(\alpha)(\phi_\lambda(G))\\
&=X_i(\alpha)\left(\frac{1}{\lambda}G(\lambda x,\mu y)\right)\\
&=\frac{1}{\lambda_i} \lambda_1^{\alpha_1}\ldots \lambda_p^{\alpha_p}\mu_1^{\alpha_{p+1}}\ldots \mu_q^{\alpha_{p+q}}X_i(\alpha)(G).
\end{align*}
So this action is given by $\phi_\lambda(X_i(\alpha))=\lambda^\alpha X_i(\alpha)$. Consequently, denoting by $\h_\calB$ the algebra generated 
by the elements $X_i(\alpha)$, it gives it a $\N^{p+q}$-graduation, for which $X_i(\alpha)$ is homogeneous of degree $\alpha$:
this graduation is finite-dimensional and connected.\\

We define a coproduct $\Delta:\h_\calB\longrightarrow \widehat{\h_\calB\otimes \h_\calB}$ in the following way:
$$\forall X \in \h_\calB,\: \forall F,G \in \bfG_\calB,\:\Delta(X)(F,G)=X(F\bullet G).$$
As the torus acts by automorphisms, for all $\lambda \in T$:
$$\Delta(\phi_\lambda^*(X))(F,G)=X(\phi_\lambda(F\bullet G))=X(\phi_\lambda(F)\bullet \phi_\lambda(G))
=(\phi_\lambda^*\otimes \phi_\lambda^*)\circ \Delta(X)(F,G).$$
Hence, $\Delta$ respects the action of $T$, so respects the graduation implied by this action, and consequently is homogeneous of degree $0$.
As the graduation is finite-dimensional, $\Delta(\h_\calB)\subseteq \h_\calB\otimes \h_\calB$. As $\bfG_\calB$ is a monoid, $\h_\calB$ is a bialgebra. 
As it is connected, it is a Hopf algebra, so $\bfG_\calB$ is a group. By construction, the group of characters of $\h_\calB$ is $\bfG_\calB$. \\

By Cartier-Quillen-Milnor-Moore's theorem, the graded dual of $\h_\calB$ is the enveloping algebra of a Lie algebra $\g$,
whose basis is given by elements $f_i(\alpha)$ dual to the elements $X_i(\alpha)$. Moreover, as the composition of $\bfG_\calB$ is linear in the second
variable, the Lie bracket of $\g$ is induced by a pre-Lie product $*$; by homogeneity, for all $i,j \in [p]$, $\alpha,\beta \in \N_*^{p+q}$, there exists a scalar
$\lambda^{(i,j)}(\alpha,\beta)$ such that:
$$f_j(\beta)*f_i(\alpha)=\lambda^{(i,j)}(\alpha,\beta) f_i(\alpha+\beta).$$
Moreover, $\lambda^{(i,j)}(\alpha,\beta)$ is the coefficient of $X_j(\beta)\otimes X_i(\alpha)$ in $\Delta(X_i(\alpha+\beta))$. Direct computations give that:
$$f_j(\beta)*f_i(\alpha)=\left(\alpha_j+\sum_{j'=1}^q A_{j',j}\alpha_{j'+p}+\delta_{i,j}\right) f_i(\alpha+\beta),$$
so $\g$ is isomorphic to $\g_\calB$ as a Lie algebra. \end{proof}\\

{\bf Example.}  If $q=0$, we obtain a Faà di Bruno group of formal diffeomorphisms, with the composition.
This is the case for the SDSE $(S)_{FdB}$ in (\ref{system}), where $A=I_N$. The associated group is:
$$G=\left(\left\{(x_1(1+F_1),\ldots,x_N(1+F_N))\mid F_1,\ldots,F_N\in \K[[x_1,\ldots,x_N]]_+\right\},\circ\right).$$

\begin{prop}
Let $V_0$ be the group $(\K[[x_1,\ldots,x_p,y_1,\ldots,y_q]]_+,+)$. The group $\bfG_\calB$ acts by automorphisms on $V_0$ by:
$$\forall F\in \bfG_\calB,\: \forall P\in V_0,\: F\hookrightarrow P=P\left(F,y\left(\frac{F}{x}\right)^\calB\right).$$
For all $r\geq 0$, the group $V_0^r \rtimes \bfG_\calB$ is isomorphic to the character group of a $\N^{p+q}$-graded Hopf algebra $\h_{\calB,r}$, 
whose graded dual is the enveloping algebra of a fundamental deg1 pre-Lie algebra $\g_\calB$ with structure coefficients:
\begin{align*}
A^{(i,j)}&: \begin{array}{|c|c|c|}
\hline i\setminus j&1\ldots p&p+1\ldots p+r\\
\hline1\ldots p+r&A_j&0\\
\hline \end{array}&
A&=\left( \begin{array}{c}
I_p\\\calB
\end{array}\right).\\
b^{(i,j)}&:  \begin{array}{|c|c|}
\hline i\setminus j&1\ldots p+r\\
\hline1\ldots p+r&0\\
\hline \end{array} \end{align*}\end{prop}

\begin{proof} Let $F\in \bfG_\calB$, $P,Q \in V_0$. Obviosuly, $F\hookrightarrow (P+Q)=F\hookrightarrow P+F\hookrightarrow Q$.
Let $F,G\in \bfG_\calB$, $P\in V_0$. Then:
\begin{align*}
F\hookrightarrow (G\hookrightarrow P)&=P\left(G\left(F,y\left(\frac{F}{x}\right)^\calB\right),
y\left(\frac{F}{x}\right)^\calB\left(\frac{G\left(F,y\left(\frac{F}{x}\right)^\calB\right)}{F}\right)^\calB\right)\\
&=P\left(G\left(F,y\left(\frac{F}{x}\right)^\calB\right),y\left(\frac{G\left(F,y\left(\frac{F}{x}\right)^\calB\right)}{x}\right)^\calB\right)\\
&=G\left(F,y\left(\frac{F}{x}\right)^\calB\right)\hookrightarrow P\\
&=(F\bullet G)\hookrightarrow P.
\end{align*}
We define an action of the torus $T=(\K^*)^{p+q}$ over $V_0$ by:
$$\psi_\lambda(P)=P(\lambda_1x_1,\ldots,\lambda_px_p,\mu_1y_1,\ldots,\mu_q y_q).$$
It is easy to prove that this is an action by automorphisms, and for all $F\in \bfG_\calB$, $P\in V_0$:
$$\psi_\lambda(F\hookrightarrow P)=\phi_\lambda(F)\hookrightarrow \psi_\lambda(P).$$

A system of coordinates of the group $V_0^r\rtimes \bfG_\calB$ is given by the elements $X_i(\alpha)$ defined on $\bfG_\calB$ 
and $Y_j(\alpha)$ defined on $V_0^r$ by:
$$Y_j(\alpha):\left\{\begin{array}{rcl}
V_0^r&\longrightarrow&\K\\
(P_1,\ldots,P_r)&\longrightarrow&\mbox{coefficient of $x^\alpha $ in $P_j$}.
\end{array}\right.$$
These elements generate an algebra $\h_{\calB,r}$, containing $\h_\calB$. The action of the torus extends the graduation of $\h_\calB$
to $\h_{\calB,r}$, making a graded connected algebra. Consequently, it inherits a coproduct, dual of the composition of the group $V_0^r\rtimes \bfG_\calB$,
making it a graded connected Hopf algebra. Note that $\h_{\calB,r}$ contains $\h_\calB$, and by construction its character group is $V_0^r\rtimes \bfG_\calB$.

The composition in $V_0^r\rtimes \bfG_\calB$ is given by:
$$(P_1,\ldots,P_r,F)\bullet (Q_1,\ldots,Q_r,G)=((P_1+F\hookrightarrow Q_1,\ldots,P_r+F\hookrightarrow Q_r,F\bullet G).$$
Consequently, it is linear in the second variable; hence, the graded dual of $\h_{\calB,r}$ is the enveloping algebra of a pre-Lie algebra $\g$.
This has a basis $(f_i(\alpha))_{i\in [p],\alpha \in \N_*^{p+q}}\sqcup (g_j(\beta))_{j\in [r],\beta \in \N_*^{p+q}}$, dual of the SDSE of coordinates
$X_i(\alpha)$ and $Y_j(\beta)$. The pre-Lie product of $f_j(\beta)$ and $f_i(\alpha)$ is the same as in $\g_\calB$, and direct computations give:
\begin{align*}
f_j(\beta)*g_i(\alpha)&=A_j\cdot \alpha g_i(\alpha+\beta),&
g_j(\beta)*f_i(\alpha)&=0,&
g_j(\beta)*g_i(\alpha)&=0.
\end{align*}
So this is indeed isomorphic to a reduced deg1 pre-Lie algebra, as announced. \end{proof}\\

This last result is proved similarly:
\begin{prop}
Let $a\in \K^r$ and $b\in \K^p$. Let $V_{a,b}$ be the group $(\K[[x_1,\ldots,x_p,y_1,\ldots,y_q]]_+,+)$. 
The group $V_0^r\rtimes \bfG_\calB$ acts by automorphisms on $V_{a,b}$ by:
$$(P_1,\ldots,P_r,F)\hookrightarrow Q=Q\left(F,y\left(\frac{F}{x}\right)^\calB\right)e^{a_1P_1+\ldots+a_rP_r}\left(\frac{F_1}{x_1}\right)^{b_1}
\ldots \left(\frac{F_p}{x_p}\right)^{b_p}.$$
If $a^{(1)},\ldots,a^{(s)}\in \K^r$ and $b^{(1)},\ldots,b^{(s)} \in \K^p$, 
the group $(V_{a^{(1)},b^{(1)}}\oplus\ldots \oplus _{a^{(s)},b^{(s)}})\rtimes(V_0^r \rtimes \bfG_\calB)$ 
is isomorphic to the character group of a $\N^{p+q}$-graded Hopf algebra $\h_{\calB,r,a,b}$, 
whose graded dual is the enveloping algebra of a fundamental deg1 pre-Lie algebra $\g_\calB$ with structure coefficients:
\begin{align*}
A^{(i,j)}&: \begin{array}{|c|c|c|}
\hline i\setminus j&1\ldots p&p+1\ldots p+r+s\\
\hline1\ldots p+r&A_j&0\\
\hline \end{array}\hspace{1cm}
A=\left( \begin{array}{c}
I_p\\\calB
\end{array}\right).\\
b^{(i,j)}&:  \begin{array}{|c|c|c|c|}
\hline i\setminus j&1\ldots p&p+1\ldots+p+r&p+r+1\ldots p+r+s\\
\hline1\ldots p+r&0&0&0\\
\hline p+r+1\ldots p+r+s&b^{(i-p-r)}_j&a^{(i-p-r)}_{j-p}&0\\
\hline \end{array}\end{align*}\end{prop}

\section{SDSE associated to a family of Feynman graphs}

\subsection{Feynman graphs}

\begin{defi}
A theory of Feynman graphs $\calT$ is given by:
\begin{itemize}
\item A set $\calHE$ of types of half-edges, with an incidence rule, that is to say an involutive map $\iota:\calHE\longrightarrow \calHE$.
\item A set $\calV$ of vertex types, that is to say a set of finite multisets (in other words finite unordered sequences) 
of elements of $\calHE$, of cardinality at least $3$. 
\end{itemize}
The edges of $\calT$ are the multisets $\{t,\iota(t)\}$, where $t$ is an element of $\calHE$. The set of edges of $\calT$ is denoted by $\calE$.
\end{defi}

{\bf Examples.} \begin{enumerate}
\item In QED, $\calHE_{QED}=\{\edgephoton,\edgeelectronun,\edgeelectrondeux\}$, and the incidence rule is given by:
\begin{align*}
\edgephoton&\longleftrightarrow\edgephoton,&
\edgeelectronun&\longleftrightarrow\edgeelectrondeux.
\end{align*}
There are two edges, $\edgephoton=\{\edgephoton,\edgephoton\}$ and $\edgeelectronun=\left\{\edgeelectronun,\edgeelectrondeux\right\}$.
There is one  vertex type,  $\vertexQED=\{\edgephoton,\edgeelectronun,\edgeelectrondeux\}$.
\item In QCD, $\calHE_{QCD}=\{\edgeelectronun,\edgeelectrondeux,\edgeghostun,\edgeghostdeux,\edgegluon\}$, and:
$$\calV_{QCD}=\left\{\vertexQCDun,\vertexQCDdeux,\vertexQCDtrois,\vertexQCDquatre\right\}.$$
The incidence rule is given by:
\begin{align*}
\edgegluon&\longleftrightarrow\edgegluon,&
\edgeelectronun&\longleftrightarrow\edgeelectrondeux,&
\edgeghostun&\longleftrightarrow\edgeghostdeux.
\end{align*}
There are three edges, $\edgegluon$ (gluon), $\edgeelectronun$ (fermion) and $\edgeghostun$ (ghost).
\item Let $N\geq 3$. In $\varphi^N$, $\calE_{\varphi^N}=\{\edgephi\}$. There is only one vertex type, which is the multiset formed by $N$ copies
of $\edgephi$. There is only one edge, denoted by $\edgephi$.
\end{enumerate}

\begin{defi}
Let $\calT=(\calHE,\calV,\iota)$ be a theory of Feynman graphs. A 1PI graph $G$ of the theory $\calT$ is given by:
\begin{itemize}
\item A nonempty, finite set $HE$ of half-edges, with a map $type:HE\longrightarrow \calHE$.
\item A nonempty, finite set $V$ of vertices.
\item An incidence map for half-edges, that is to say an involution map $i:HE\longrightarrow HE$.
\item A source map for half-edges, that is to say a map $s:HE\longrightarrow V$.
\end{itemize}
The following conditions must be satisfied:
\begin{enumerate}
\item (Respect of the incidence rule) for any $e\in HE$ such that $i(e)\neq e$, $\iota(type(e))=type(i(e))$.
\item (Respect of the vertex types) for any $v \in V$, the multiset $type(v)=\{type(e)\mid s(e)=v\}$ belongs to $\calV$.
\item (Connectivity and one-particule irreducibility) the set of internal edges of $G$ is:
$$Int(G)=\{\{e,i(e)\}\mid e\in HE,i(e)\neq e\}.$$ 
The source map makes $(V,Int(G))$ a graph. This graph is 1-PI,
that is to say that it is connected and remains connected if one edge $e\in Int(G)$ is deleted.
\item (External structure) the set of external half-edges f $G$ is:
$$Ext(G)=\{e \mid e\in HE, i(e)=e\}.$$ 
We define $typeExt(G)$ as the multiset $\{type(e)\mid e\in Ext(G)\}$. Two case are possible:
\begin{enumerate}
\item $typeExt(G)=\{t_1,t_2\}$, with $\iota(t_1)=t_2$. In this case, we shall say that the external structure of $G$ is of type edge $typeExt(G)$.
\item  $typeExt(G)\in \calV$. In this case, we shall say that the external structure of $G$ is of type vertex $typeExt(G)$.
\end{enumerate} \end{enumerate}
A Feynman graph is the disjoint union of a finite number (possibly $0$) 1-PI Feynman graphs, called its connected components.
The set of Feynman graphs of the theory $\calT$ is denoted by $\fg_\calT$.
\end{defi}

We shall only consider theories such that there exists 1-PI Feynman graphs for all type of external structures.\\

{\bf Examples.} \begin{enumerate}
\item Here are examples of 1-PI Feynman graphs in QED.
$$\begin{array}{|c|c|}
\hline \mbox{External structure}&\mbox{Examples}\\
\hline&\\ \vertexQED&\QEDun,\QEDcinq,\QEDsept,\QEDhuit,\QEDdix,\QEDonze\\  &\\
\hline& \\ \edgephoton&\QEDdeux,\QEDquatre,\QEDquatorze\\
\hline \edgeelectronun&\QEDtrois,\QEDdouze,\QEDtreize\\
\hline \end{array}$$
\item Here are examples of 1-PI Feynman graphs in QCD.
$$\begin{array}{|c|c|}
\hline \mbox{External structure}&\mbox{Examples}\\
\hline&\\ \vertexQCDun&\QCDun,\QCDsept,\QCDdix,\QCDonzebis\\ &\\
\hline&\\ \vertexQCDdeux&\QCDunbis,\QCDcinq,\QCDhuit,\QCDonze\\ &\\
\hline&\\ \vertexQCDtrois&\QCDunter\\ &\\
\hline&\\ \vertexQCDquatre&\QCDquinze,\QCDseize\\ &\\
\hline&\\ \edgegluon&\QCDdeux,\QCDdeuxbis,\QCDquatre,\QCDquatorze\\ 
\hline \end{array}$$
$$\begin{array}{|c|c|}
\hline \mbox{External structure}&\mbox{Examples}\\
\hline&\\ \edgeelectronun&\QCDtrois,\QCDdouze,\QCDtreize\\
\hline&\\ \edgeghostun&\QCDtroisbis,\QCDdouzebis,\QCDtreizebis\\
\hline \end{array}$$
\item Here are examples of 1-PI Feynman graphs in $\varphi^3$.
$$\begin{array}{|c|c|}
\hline \mbox{External structure}&\mbox{Examples}\\
\hline&\\ \vertexphi&\phiun,\phicinq,\phisept,\phihuit,\phidix,\phionze\\ &\\
\hline&\\  \edgephi&\phideux,\phiquatre,\phiquatorze\\
\hline \end{array}$$
\end{enumerate}

\begin{defi}\begin{enumerate}
\item Let $G=(HE,V,i,s)$ and $G'=(HE',V',i',s')$ be two Feynman graphs of a theory $\calT$. We shall say that $G'$ is a subgraph of $G$ if:
\begin{enumerate}
\item $HE'\subseteq HE$, $V'=s(HE')$ and $s'=s_{\mid HE'}$.
\item For any $e\in HE'$, $i'(e)=e$ or $i'(e)=i(e)$.
\end{enumerate}
\item Let $G'$ be a connected subgraph of $G$. We define a structure $G/G'=(HE'',V'',i'',s'')$ in the following way:
\begin{itemize}
\item If the type of the external structure of $G'$ is a vertex:
\begin{enumerate}
\item $HE''=(HE\setminus HE')\sqcup Ext(G')$.
\item $V''=(V\setminus V') \sqcup \{0\}$.
\item For all $e\in HE''$:
$$s''(e)=\begin{cases}
s(e)\mbox{ if }e\in HE\setminus HE',\\
0\mbox{ if }e\in Ext(G').
\end{cases}$$
\item For all $e\in HE''$, $i''(e)=i(e)$.
\end{enumerate}
\item If the type of the external structure of $G'$ is an edge, let us denote by $e_1$ and $e_2$ its two external half-edges.
\begin{enumerate}
\item $HE''=HE\setminus HE'$.
\item $V''=V\setminus V'$.
\item For all $e\in HE''$, $s''(e)=s(e)$.
\item For all $e\in HE''$:
$$i''(e)=\begin{cases}
i(e_2)\mbox{ if }e=i(e_1),\\
i(e_1)\mbox{ if}e=i(e_2),\\
i(e)\mbox{ otherwise.}
\end{cases}$$
\end{enumerate} \end{itemize} 
If $G'$ is not connected, we put $G'=G'_1\ldots G'_k$ its decomposition into connected parts, and define $G/G'=(\ldots (G/G'_1)/G'_2)\ldots)/G'_k$.
It does not depend of the order chosen on the connected components of $G'$.
\item If $G/G'$ is a Feynman graph, we shall say that $G'$ is an admissible subgraph and we shall write $G'\subseteq G$.
\end{enumerate} \end{defi}

Roughly speaking, $G/G'$ is obtained by deleting $G'$ from $G$ and contracting the hole which appeared until it vanishes.
By convention, $G/G=1$ and $G/1=G$. Observe that if $G'\subsetneq G$ and $G$ is 1-PI, then $G/G'$ is also 1-PI, with the same 
external structure as $G$.\\

The set $\fg(\calT)$ is a basis of the Hopf algebra $\algfg$ associated to a theory $\calT$ of Feynman graphs.
Its product is given by the disjoint union of Feynman graphs; its coproduct is given by:
$$\forall G\in \fg(\calT),\: \Delta(G)=\sum_{G'\subseteq G} G'\otimes G/G'.$$

{\bf Examples.} In QED:
\begin{align*}
\Delta \QEDonze&=\QEDonze\otimes 1+1\otimes \QEDonze+\QEDdeux\otimes\QEDun,\\
\Delta \QEDcinq&=\QEDcinq\otimes 1+1\otimes \QEDcinq+\QEDun\otimes\QEDun,\\
\Delta \QEDquatre&=\QEDquatre\otimes 1+1\otimes \QEDquatre+2 \QEDun\otimes \QEDdeux.
\end{align*}

\begin{defi} 
Let $G$ be a Feynman graph of a given theory $\calT$. The loop number of $G$ is:
$$\ell(G)=\sharp Int(G)-\sharp Vert(G)+\sharp\{\mbox{connected components of $G$}\}.$$
Note that because of the $1$-PI condition, for all nonempty graph $G$, $\ell(G)\geq 1$.
\end{defi}
We shall prove afterwards that the loop number defines a connected $\N$-graduation of the Hopf algebra $\algfg$.

\subsection{Graduations}

Let us fix a theory $\calT=(\calHE,\calV,\iota)$. We look for graduations of the Hopf algebra $\algfg$. We shall use the following notions:

\begin{defi}\label{deficombi}
\begin{enumerate}
\item The incidence matrix of $\calT$ is the matrix $A_\calT=(a_{e,v})_{e\in \calHE,v\in\calV}$, where $a_{e,v}$ is the multiplicity of $e$
in the multiset $v$. 
\item The reduced incidence matrix of $\calT$ is the matrix $A'_\calT=(a'_{e,v})_{e\in \calE, v\in \calV}$, where:
$$a'_{e,v}=\begin{cases}
\displaystyle \frac{a_{e_1,v}}{2}\mbox{ if } e=\{e_1,e_1\},\\[2mm]
\displaystyle \frac{a_{e_1,v}+a_{e_2,v}}{2}\mbox{ if } e=\{e_1,e_2\} \mbox{ with }e_1\neq e_2.
\end{cases}$$ 
\item Let $G\in \fg(\calT)$. We define four vectors related to $G$:
\begin{enumerate}
\item $V_G=(v_t(G))_{t\in \calV}$, where $v_t(G)$ is the number of vertices $v$ of $G$ such that $type(v)=t$.
\item $E_G=(he_t(G))_{t\in \calHE}$, where $he_t(G)$ is the number of half-edges $e$ of $G$ such that $type(e)=t$.
\item $E'_G=(e_t(G))_{t\in \calE}$, where $e_t(G)$ is the number of internal edges $e$ of $G$ such that $type(e)=t$.
\item $S_G=(s_t(G))_{t\in \calV\sqcup\calE}$, where $s_t(G)$ is the number of connected components of $G$ of external structure $t$.
\end{enumerate}
\end{enumerate}\end{defi}

{\bf Examples}. 
\begin{align*}
A_{QED}&=\left(\begin{array}{c}
1\\1\\1
\end{array}\right),&
A_{QCD}&=\left(\begin{array}{cccc}
1&0&0&0\\1&0&0&0\\0&1&0&0\\0&1&0&0\\1&1&3&4
\end{array}\right),&
A_{\varphi^n}&=(n). \\
A'_{QED}&=\left(\begin{array}{c}
1\\ \frac{1}{2}
\end{array}\right),&
A'_{QCD}&=\left(\begin{array}{cccc}
1&0&0&0\\ 0&1&0&0\\ \frac{1}{2}&\frac{1}{2}&\frac{3}{2}&2
\end{array}\right),&
A'_{\varphi_n}&=\left(\frac{n}{2}\right).
\end{align*}

\begin{prop}\label{prop37} Let $G\in \fg(\calT)$. Then:
\begin{enumerate}
\item $E_G=A_\calT V_G$.
\item $E'_G=A'_\calT V_G-(A'_\calT\: Id)S_G$.
\item The number of external half-edges of $G$ is $(1\ldots 1)(A_\calT\: 2Id)S_G$.
\item The loop number of $G$ is:
$$\ell(G)=\left(\frac{(1\ldots 1)A_\tau}{2}-(1\ldots 1)\right)V_G-\left((1\ldots 1)\left(\frac{A_\calT}{2}\: 0\right)-(1\ldots 10\ldots 0)\right)S_G.$$
\end{enumerate}\end{prop}

\begin{proof} The first three points are easy results of graph theory. 
The number of connected components of $G$ is $(1\ldots 1)S_G$; the number of external half-edges of $G$ is given, from the third point, by
$(1\ldots 1)(A_\calT\: 2Id)S_G$. Hence, the number of internal edges of $G$ is given by:
$$\frac{(1\ldots 1)E_G-(1\ldots 1)(A_\calT\: 2Id)S_G}{2}.$$
The loop number of $G$ is consequently given by:
\begin{align*}
\ell(G)&=\frac{(1\ldots 1)E_G-(1\ldots 1)(A_\calT\: 2Id)S_G}{2}-(1\ldots 1)V_G+(1\ldots 1)S_G\\
&=\left(\frac{(1\ldots 1)A_\tau}{2}-(1\ldots 1)\right)V_G-\left((1\ldots 1)\left(\frac{A_\calT}{2}\: Id\right)-(1\ldots 1)\right)S_G\\
&=\left(\frac{(1\ldots 1)A_\tau}{2}-(1\ldots 1)\right)V_G-\left((1\ldots 1)\left(\frac{A_\calT}{2}\: 0\right)-(1\ldots 10\ldots 0)\right)S_G,
\end{align*}
which proves the last point. \end{proof}\\

We now look for $\mathbb{Q}^N$-graduations of the Hopf algebra $\algfg$, which only depend on the combinatorial datas
of definition \ref{deficombi}-3. According to proposition \ref{prop37}, for such a graduation,
there exists a map $f:\N^{|\calV|}\times \N^{|\calV|+|\calE|}\longrightarrow\mathbb{Q}^N$, such that for any graph $G$, $deg(G)=f(V_G,S_G)$.

\begin{prop} \label{propgraduation}
Let $f:\N^{|\calV|}\times \N^{|\calV|+|\calE|}\longrightarrow\mathbb{Q}^N$. We consider the $\mathbb{Q}^N$-graduation of $\algfg$ defined
by $deg(G)=f(V_G,S_G)$. It is a Hopf algebra graduation if, and only if, there exists $C \in M_{N,|\calV|}(\mathbb{Q})$ such that
for any Feynman graph $G$: $$deg(G)=CV_G-(C\: 0)S_G.$$
\end{prop}

\begin{proof} Let $G$ and $G'$ be two graphs. Then $V_{GG'}=V_G+V_{G'}$ and $S_{GG'}=S_{G}+S_{G'}$. 
Consequently, the graduation respects the product if, and only if, for all $G,G'$:
$$f(V_G+V_{G'},S_G+S_{G'})=f(V_G,S_G)+f(V_{G'},S_{G'}).$$
that is to say if, and only if, $f$ is additive. Hence, $f$ gives a graduation of the algebra $\algfg$ if, and only if,
there exists $C\in M_{N,|\calV|}(\mathbb{Q})$ and $D\in M_{N,|\calV|+|\calE|}(\mathbb{Q})$ such that for any Feynman graph $G$, $deg(G)=CV_G+DS_G$.\\

Let $G'\subseteq G$. By definition of $G''=G/G'$, $V_{G''}=V_G-V_{G'}+(Id\: 0)S_{G'}$ and $S_{G''}=S_G$. Hence:
\begin{align*}
deg(G)&=CV_G+DS_G\\
&=CV_{G''}+CV_{G'}-C(Id\:0)S_{G'}+DS_{G''}\\
&=deg(G')+deg(G'')-(D+C(Id \: 0))S_{G'}.
\end{align*}
So $f$ gives a graduation of $\algfg$ if, and only if, for all subdiagram $G'\subseteq G$, $(D+C(Id \: 0))S_{G'}=0$.
As there exists diagrams for any external structure, we can choose $G$ and $G'$ such that $S_{G'}$ is the $i$-th vector of the canonical basis;
hence, we have a graduation of $\algfg$ if, and only if, $D+C(Id \: 0)=0$. \end{proof}\\

Consequently, any matrix $C \in M_{N,|\calV|}(\mathbb{Q})$ defines a $\mathbb{Q}^N$-graduation of the Hopf algebra $\algfg$.
This of course may be not a $\N^N$-graduation, or may be not connected. \\

{\bf Examples.} \begin{enumerate}
\item The loop number $\ell$ gives a Hopf algebra $\N$-graduation with $\displaystyle C_\ell=\frac{(1\ldots 1)A_\tau}{2}-(1\ldots 1)$.
This is a connected $\N$-graduation, as we only consider 1-PI graphs.
\item Let $t \in \calE_\calT$, and let $C$ be the $t$-th row of $A'_\calT$; the associated graduation is noted $deg_t$.
For all $G\in \fg(\calT)$, $deg_t(G)=e_t(G)+s_t(G)$. This is a $\N$-graduation, which may be not connected.
\item If $deg$ and $deg'$ are two graduations of the Hopf algebra $\algfg$, then $deg\oplus deg'$ defined by $deg\oplus deg'(G)=(deg(G),deg'(G))$
is also a graduation of $\algfg$. If $deg$ and $deg'$ are respectively given by $C$ and $C'$, 
$deg\oplus deg'$ is given by $\left(\begin{array}{c}C\\C'\end{array}\right)$.
\end{enumerate}

\subsection{Insertions}

\begin{defi}
Let $G$ and $G'$ be two Feynman graphs of a theory $\calT$.
\begin{enumerate}
\item  We denote by $G'_1,\ldots,G'_k$ the connected components of $G'$. 
A place of insertion $f$of  $G'$ into $G$ is given by:
\begin{enumerate}
\item for all $G'_i$ of external structure of type a vertex $t$, a pair $(v_i,f_i)$, where $v_i$ is a vertex of $G$ of type $t$, and $f_i$ a bijection 
from the set of external edges of $\bfG_i$ to the set of half-edges $e$ of $G'$ such that $s(e)=t$, compatible with the type, that is to say 
$type(f_i(e'))=type(e')$ for all $e'$. Moreover, if $G'_i$ and $G'_j$ are both of external structure of type $t$, with $i\neq j$, then $v_i\neq v_j$.
\item for all $G'_i$ of external structure of type an edge $t$, a pair $(e_i,f_i)$ where $e_i=\{e_i^{(1)},e_i^{(2)}\}$ is an internal edge of $G$
of type $t$, and $f_i$ a bijection from the set of the two external half-edges of $t$ into $\{e_i^{(1)},e_i^{(2)}\}$.
\item For all internal edge $e$ of $G$, the set of components $\bfG_i$ such that $e_i=e$ is totally ordered.
\end{enumerate}
Note that the set of places of insertion of $G'$ into $G$ is finite and may be empty. Its cardinality is denoted $ins(G',G)$.
\item Let $F$ be a place of insertion of $G'$ into $G$. The insertion $G' \insere{F}G$ is the Feynman graph obtained in this way:
\begin{enumerate}
\item For all $G'_i$ of external structure of type vertex, delete $v_i$ and all the half-edges $e$ such that $s(e)=v_i$; then glue
each external edge $e'$ of $G'_i$ to $i(f_i(e'))$ if if is not equal to $f_i(e')$; otherwise, $e'$ becomes an external edge.
\item For each internal edge $e$, such that there exists components $G'_i$ with $e_i=e$, first separate the two half-edges constituing this internal edge; 
then insert all these components $G'_i$, following their total order, by gluing their external edges with the two open half-edges according to $f_i$.
\end{enumerate}\end{enumerate}\end{defi}

For any Feynman graph such that $ins(G',G)\neq 0$, we put:
$$B_G(G')=\frac{1}{ins(G',G)} \sum_F G'\insere{F}G.$$

\begin{prop}\begin{enumerate}
\item For all graph $G$, the space $I_G=Vect(G',ins(G',G)\neq 0)$ is a left comodule.
\item For all primitive graph $G$, for all $x \in I_G$:
$$\Delta\circ B_G(x)=B_G(x)\otimes 1+(Id \otimes B_G)\circ \Delta(x).$$
\item We define a graduation on $\algfg$ with the help of a matrix $C \in M_{N,|\calV|}(\mathbb{Q})$. Then for any Feynman graph $G$,
$B_G$ is homogeneous of degree $deg(G)$.
\end{enumerate}\end{prop}

\begin{proof}  1. Let $G$ and $G'$ be two graphs. Then $G'\in I_G$ if, and only if, the two following conditions hold:
\begin{itemize}
\item For any $t\in \calV$, $s_t(G')\leq v_t(G)$.
\item For any $t\in \calE$, $(s_t(G') \geq 1)\Longrightarrow (e_t(G')\geq 1)$.
\end{itemize}
Consequently, if $G'\in I_G$ and $G''\subset G'$, noting that $s_t(G'/G'') \leq s_t(G')$ for all $t\in \calV \sqcup \calE$,  then $G'/G'' \in I_G$. 
So $I_G$ is a left comodule. \\

2. Let $G' \in \fg(\calT)$, such that $ins(G',G)\neq 0$. As $G$ is primitive, $\Delta(G)=G=\otimes 1+1\otimes G$, so $G$ has no proper subgraph. 
For all insertion place $f$, let us consider a subgraph $H$ of $G''=G'\insere{f}G$. If $H$ contains internal edges of $G''$ which does not belong to $G'$, 
as $G'$ has no proper subgraph, it contains all the edges of $G$, and, as $H$ is a subgraph, it is equal to $G$.
Otherwise, $H$ is a subgraph of $G'$, and then $G''/H=G'/H \insere{f'} G$ for a particular $F'$. Summing, we obtain:
\begin{align*}
\Delta(B_G(G'))&=\frac{1}{ins(G',G)} \sum_{f} \left(G'\insere{f}G\otimes 1+\sum_{H\subseteq G'}
H' \otimes G'/H\insere{f'}G\right)\\
&=B_G(G')\otimes 1+\sum_{H\subseteq G'} H\otimes B_G(G'/H)\\
&=B_G(G')\otimes 1+(Id\otimes B_G)\circ \Delta(G').
\end{align*} 
By linearity, the result holds for all $x \in I_G$. \\

3. Let $G$ and $G'$ be Feynman graphs, and $G''=G'\insere{F} G$. Then $S_{G''}=S_G$ and $V_{G''}=V_G+V_{G'}-(Id \: 0)S_{G'}$.
Hence:
\begin{align*}
deg(G'')&=CV_G+CV_{G'}-C(Id\:0)S_{G'}-(C\: 0)S_G\\
&=CV_G-(C\: 0)S_G+CV_{G'}-(C\:0)S_{G'}\\
&=deg(G)+deg(G').
\end{align*}
So $B_G$ is homogeneous of degree $deg(G)$. \end{proof}

\subsection{SDSE associated to a theory of Feynman graphs}

Let $\calT$ be a family of Feynman graphs. We put $\calV=\{t_1,\ldots,t_k\}$, $\calE=\{t_{k+1},\ldots,t_{k+l}\}$ and $M=k+l$.
We choose a connected $\N^N$-graduation of $\algfg$ given by a $N\times k$ matrix $C$. In order to ease the notation,
for all $i\in [k]$, we put $v_{t_i}(G)=v_i(G)$ and for all $k+1\leq j\leq k+l$, $e_{t_j}(G)=e_j(G)$, for any Feynman graph $G$. \\

{\bf Notations.}\begin{enumerate}
\item For each $i$, we denote by $P_i$ the set of primitive 1-PI Feynman graphs of the theory $\calT$ of external structure of type $t_i$.
\item Let $\alpha_1,\ldots,\alpha_N$ be $N$ indeterminates (the coupling constants). For any graph $G$, if $deg(G)=(d_1,\ldots,d_N)$,
we put $\alpha^{deg(G)}=\alpha_1^{d_1}\ldots \alpha_N^{d_N}$.
\end{enumerate}

We consider the following SDSE on $\algfg$: 
\begin{align*}
(S_\calT):\: \forall i\in [M],\: X_i&=\sum_{G\in P_i} \alpha^{deg(G)} B_G\left(\prod_{j=1}^k (1+X_j)^{v_i(G)}\prod_{j=k+1}^{k+l}(1-X_j)^{-e_j(G)}\right).
\end{align*}
We decompose $X_i$ according to the powers of the $\alpha_i$:
$$X_i=\sum_{d \in \NN} \alpha^dX_i(d).$$
It is not difficult to show that $X_i(d)$ is homogeneous of degree $d$, as $B_G$ is homogeneous of degree $deg(G)$.
The subalgebra generated by the $X_i(d)'s$ is denoted by $\h_{(S_\calT)}$. \\

Combinatorially, $X_i$ is a span of all connected graph of external structure of type $t_i$; its homogeneous components can be inductively computed
by taking all possible insertions of already computed homogeneous components of $X_j$ into primitive Feynman graphs of the good external structure,
in order to obtain the expected degree.\\

We lift this SDSE to the level of rooted trees. The set of decorations is the set of primitive connected Feynman graphs:
$$P=\bigsqcup_{i=1}^{k+l} P_i.$$
The graduation of $\h_{CK}^D$ is given by the degree of primitive Feynman graphs, and we consider the SDSE  on $\h_{CK}^D$:
\begin{align*}
(S'_\calT):\: \forall i\in [M],\:Y_i&=\sum_{G\in P_i} B_G\left(\prod_{j=1}^k (1+Y_j)^{v_i(G)}\prod_{j=k+1}^{k+l}(1-Y_j)^{-e_j(G)}\right).
\end{align*}
The homogeneous component of $Y_i$ of degree $d$ is denoted by $Y_i(d)$ and the subalgebra of $\h_{CK}^D$ generated by the $Y_i(d)$
is denoted by  $\h_{(S'_\calT)}$. 

\begin{prop}
If $\h_{(S'_\calT)}$ is a Hopf subalgebra of $\h_{CK}^D$, then $\h_{(S_\calT)}$ is a Hopf subalgebra of $\algfg$, and the algebra morphism defined by
$Y_i(d)\longrightarrow X_i(d)$ is a surjective Hopf algebra morphism from $\h_{(S'_\calT)}$ to $\h_{(S_\calT)}$.
\end{prop}

\begin{proof} Let $T$ be a rooted tree decorated by $D$. We shall say it is admissible if for all vertex $v$ of $T$, denoting by $G$ the decoration of $v$
and by $\bfG_1,\ldots,\bfG_k$ the decorations of the children of $v$, then $\bfG_1\ldots \bfG_k \in I_G$.
We denote by $A'$ the subalgebra generated by all admissible trees. 
If $T$ is admissible, then for all admissible cut $c$ of $T$, $P^c(T)$ and $R^c(T)$ are admissible, so $A'$ is a Hopf subalgebra.
By definition of $(S_\calT)$, $Y_i(d) \in A'$ for all $i \in [M]$, $d\in \NN$.

One can define an algebra morphism $\phi$ from $A'$ to $\algfg$ inductively by:
$$\phi(B^+_G(T_1\ldots T_k))=B_G(\phi(T_1)\ldots \phi(T_k)),$$
for all admissible tree $B^+_G(T_1\ldots T_k)$. It is well-defined: indeed, if $\phi(T_1),\ldots,\phi(T_k)$ are well-defined,
then for all $i$, $\phi(T_i)$ is a linear span of graphs with the external structure given by the decoration of the root of $T_i$.
As $B_G(T_1\ldots T_k)$ is admissible, $\phi(T_1)\ldots \phi(T_k) \in I_G$, so $\phi(B_G^+(T_1\ldots T_k))$ is well-defined.
As $B_G$ and $B_G^+$ are both homogeneous of degree $deg(G)$, an easy induction proves that $\phi$ is homogeneous of degree $0$.
As $\phi \circ B^+_G=B_G\circ \phi$ on $A'$ for all $G'$, $\phi(Y_i(d))=X_i(d)$ for all $i\in [M]$ and all $d\in \NN$. 
By the one-cocycle property of $B^+_G$ and $B_G$ on $I_G$, $(\phi\otimes \phi)\circ \Delta=\Delta\circ \phi$ on $A'$.
Consequently, if $\h_{(S'_\calT)}$ is a Hopf subalgebra of $\h_{CK}^D$, its image $\h_{(S_\calT)}$ is a Hopf subalgebra of $\algfg$. \end{proof}

\begin{theo} \label{theorangC}
If $Rank(C)=|\calV|$, then $\h_{(S'_\calT)}$ is a Hopf subalgebra of $\h_{CK}^D$;
moreover, the SDSE $(S'_\calT)$ is associated to a deg1 pre-Lie algebra.
\end{theo}

\begin{proof}  First, observe that, as $C$ is a $N\times k$-matrix, $Rank(C)\leq k=|\calV|$.\\

 Let us assume that $Rank(C)=k$. There exists a matrix $C' \in M_{k,N}(\K)$, such that $C'C=Id_k$. 
For any primitive Feynman graph $G$ of external structure $t_i$ and of degree $d$, if $(\epsilon_1,\ldots,\epsilon_M)$ is the canonical basis of $\K^M$,
noting that $d=CV_G-(C\: 0)\epsilon_i$:
\begin{align*}
V_G&=C'd+(Id_k\: 0)\epsilon_i,&E'_G&=A'_\calT C'd-(0\: Id_l)\epsilon_i.
\end{align*}
For all $i\in [M]$, for all $d\in \NN$, we put:
$$B^+_{i,n}=\sum_{G\in P_i, deg(G)=d} B^+_G.$$
The SDSE can be written as:
\begin{align*}
Y_i&=\sum_{d\in \NN} B_{i,d}\left(\prod_{j=1}^k(1+Y_j)^{\sum_p c'_{j,p}d_p} \prod_{j=k+1}^{k+l}(1-Y_j)^{\sum_{p,q}a'_{j,p}c'_{p,q}d_j}(1+Y_i)
\right)\mbox{ if }i\leq k,\\
Y_i&=\sum_{d\in \NN} B_{i,d}\left(\prod_{j=1}^k(1+Y_j)^{\sum_p c'_{j,p}d_p} \prod_{j=k+1}^{k+l}(1-Y_j)^{\sum_{p,q}a'_{j,p}c'_{p,q}d_j}(1-Y_i)
\right)\mbox{ if }i\geq k+1.
\end{align*}
Hence, we recognize the deg1 pre-Lie algebra with $I_p=\{p\}$ for all $1\leq p\leq k+l$, $b$ given by:
$$b^{(i,j)}=\begin{cases}
\delta_{i,j}\mbox{ if }i\leq k,\\
-\delta_{i,j}\mbox{ if }i\geq k+1,
\end{cases}$$
and $A$ given by the matrix $\left(\begin{array}{c} C'\\-A'_\calT C'\end{array}\right)$. \end{proof}\\

As $C'$ has also rank $k$:

\begin{cor}
If $Rank(C)=|\calV|=k$, the graded dual of the Hopf algebra $\h_{(S'_\calT)}$
 is the enveloping algebra of the reduced deg1 pre-Lie algebra with structure coefficients given by:
\begin{align*}
A'^{(i,j)}&: \begin{array}{|c|c|c|}
\hline i\setminus j&1\ldots k&k+1\ldots k+l\\
\hline1\ldots M&A'_j&0\\
\hline \end{array}&A'&=\left(\begin{array}{c}
I_k\\ A''
\end{array}\right)\\
b'^{(i,j)}&:  \begin{array}{|c|c|}
\hline i\setminus j&1\ldots M\\
\hline 1\ldots M&0\\
\hline \end{array} \end{align*}
If $C$ is invertible, then $A''=-A'_\calT$. 
Moreover, $\h_{(S'_\calT)}$ is isomorphic to the coordinate Hopf algebra of the group $V_0^l \rtimes \bfG_{A''}$.
\end{cor}

\begin{proof} It remains to consider the case where $C$ is invertible. In this case, $C'=C^{-1}$ and
$A=\left(\begin{array}{c} C'\\-A'_\calT C'\end{array}\right)$. We then take $A=A'C=\left(\begin{array}{c}Id_k\\-A_\calT'\end{array}\right)$. \end{proof}\\

{\bf Examples.}
\begin{enumerate}
\item If there is only one vertex type, we can choose the graduation by the loop number.
\begin{enumerate}
\item For QED, $C=\left(\frac{1}{2}\right)$, so $C'=(2)$; hence, $A=\left(\begin{array}{c}\frac{1}{2}\\-\frac{1}{2}\\ -\frac{-1}{4}\end{array}\right)$
and $A'=\left(\begin{array}{c}1\\-1\\-\frac{1}{2}\end{array}\right)$. The SDSE is:
\begin{align*}
X_1&=\sum_{k\geq 1}\alpha^k \sum_{G\in D_1(k)} B_G\left(\frac{(1+X_1)^{2k+1}}{(1-X_2)^k(1-X_3)^{2k}}\right),\\
X_2&=\sum_{k\geq 1}\alpha^k \sum_{G\in D_2(k)} B_G\left(\frac{(1+X_1)^{2k}}{(1-X_2)^{k-1}(1-X_3)^{2k}}\right),\\
X_3&=\sum_{k\geq 1}\alpha^k \sum_{G\in D_3(k)} B_G\left(\frac{(1+X_1)^{2k}}{(1-X_2)^k(1-X_3)^{2k-1}}\right),
\end{align*}
where $D_1(k)$, $D_2(k)$ and $D_3(k)$ are sets of primitive Feynman graphs with $k$ loops and respective external structures $\vertexQED$, 
$\edgephoton$ and $\edgeelectronun$. In particular:
\begin{align*}
D_2(k)&=\begin{cases}
\left\{\QEDdeux\right\}\mbox{ if }k=1,\\
\emptyset \mbox{ otherwise};
\end{cases}&
D_3(k)&=\begin{cases}
\left\{\QEDtrois\right\}\mbox{ if }k=1,\\
\emptyset \mbox{ otherwise}.
\end{cases}
\end{align*}
\item In $\varphi^n$, $C=\left(\frac{n-2}{2}\right)$, so $C'=\left(\frac{2}{n-2}\right)$; 
hence, $A=\left(\begin{array}{c}\frac{2}{n-2}\\-\frac{n}{n-2}\end{array}\right)$ and $A'=\left(\begin{array}{c}1\\-\frac{n}{2}\end{array}\right)$.
The SDSE is:
\begin{align*}
X_1&=\sum_{k\geq 1}\alpha^k \sum_{G\in D_1(k)} B_G\left(\frac{(1+X_1)^{\frac{2k}{n-2}+1}}{(1-X_2)^{\frac{nk}{n-2}}}\right),\\
X_2&=\sum_{k\geq 1}\alpha^k \sum_{G\in D_2(k)} B_G\left(\frac{(1+X_1)^{\frac{2k}{n-2}}}{(1-X_2)^{\frac{nk}{n-2}-1}}\right),
\end{align*}
where $D_1(k)$ and $D_2(k)$ are sets of primitive Feynman graphs with $k$ loops and respective external structures the vertex and the edge.
\end{enumerate}
\item In QCD, we take:
$$C=\left(\begin{array}{cccc}
1&0&0&0\\
0&1&0&0\\
\frac{1}{2}&\frac{1}{2}&\frac{3}{2}&2\\
\frac{1}{2}&\frac{1}{2}&\frac{1}{2}&1
\end{array}\right).$$
If $G$ is a QCD Feynman graph, then:
\begin{align*}
deg(G)&=\left(deg_{\edgeelectronun}(G),\:deg_{\edgeghostun}(G),\:deg_{\edgegluon}(G),\:\ell(G)\right).
\end{align*}
It is a connected $\mathbb{N}^4$-graduation. Moreover, $C'=C^{-1}$, and:
\begin{align*}
A&=\left(\begin{array}{cccc}
1&0&0&0\\0&1&0&0\\1&1&2&-4\\-1&-1&-1&3\\
-1&0&0&0\\0&-1&0&0\\0&0&-1&0
\end{array}\right),&
A'&=\left(\begin{array}{cccc}
1&0&0&0\\0&1&0&0\\0&0&1&0\\0&0&0&1\\
-1&0&0&0\\0&-1&0&0\\-\frac{1}{2}&-\frac{1}{2}&-\frac{3}{2}&-2\\
\end{array}\right). \end{align*}
The SDSE is:
\begin{align*}
X_1&=\sum_{k\in \mathbb{N}^4_*}
\alpha^k \sum_{G\in D_1(k)} B_G\left(\frac{(1+X_1)^{k_1+1}(1+X_2)^{k_2}(1+X_3)^{\alpha(k)}(1+X_4)^{\beta(k)}}
{(1-X_5)^{k_1}(1-X_6)^{k_2}(1-X_7)^{k_3}}\right),\\
X_2&=\sum_{k\in \mathbb{N}^4_*}
\alpha^k \sum_{G\in D_2(k)} B_G\left(\frac{(1+X_1)^{k_1}(1+X_2)^{k_2+1}(1+X_3)^{\alpha(k)}(1+X_4)^{\beta(k)}}
{(1-X_5)^{k_1}(1-X_6)^{k_2}(1-X_7)^{k_3}}\right),\\
X_3&=\sum_{k\in \mathbb{N}^4_*}
\alpha^k \sum_{G\in D_3(k)} B_G\left(\frac{(1+X_1)^{k_1}(1+X_2)^{k_2}(1+X_3)^{\alpha(k)+1}(1+X_4)^{\beta(k)}}
{(1-X_5)^{k_1}(1-X_6)^{k_2}(1-X_7)^{k_3}}\right),\\
X_4&=\sum_{k\in \mathbb{N}^4_*}
\alpha^k \sum_{G\in D_4(k)} B_G\left(\frac{(1+X_1)^{k_1}(1+X_2)^{k_2}(1+X_3)^{\alpha(k)}(1+X_4)^{\beta(k)+1}}
{(1-X_5)^{k_1}(1-X_6)^{k_2}(1-X_7)^{k_3}}\right),\\
X_5&=\sum_{k\in \mathbb{N}^4_*}
\alpha^k \sum_{G\in D_5(k)} B_G\left(\frac{(1+X_1)^{k_1}(1+X_2)^{k_2}(1+X_3)^{\alpha(k)}(1+X_4)^{\beta(k)}}
{(1-X_5)^{k_1-1}(1-X_6)^{k_2}(1-X_7)^{k_3}}\right),\\
X_6&=\sum_{k\in \mathbb{N}^4_*}
\alpha^k \sum_{G\in D_6(k)} B_G\left(\frac{(1+X_1)^{k_1}(1+X_2)^{k_2}(1+X_3)^{\alpha(k)}(1+X_4)^{\beta(k)}}
{(1-X_5)^{k_1}(1-X_6)^{k_2-1}(1-X_7)^{k_3}}\right),\\
X_7&=\sum_{k\in \mathbb{N}^4_*}
\alpha^k \sum_{G\in D_7(k)} B_G\left(\frac{(1+X_1)^{k_1}(1+X_2)^{k_2}(1+X_3)^{\alpha(k)}(1+X_4)^{\beta(k)}}
{(1-X_5)^{k_1}(1-X_6)^{k_2}(1-X_7)^{k_3-1}}\right),
\end{align*}
with $\alpha(k)=k_1+k_2+2k_3-4k_4$ and $\beta(k)=-k_1-k_2-k_3+3k_4$, and
where $D_1(k)$, $D_2(k)$, $D_3(k)$, $D_4(k)$, $D_5(k)$, $D_6(k)$ and $D_7(k)$ are sets of primitive Feynman graphs of degree $k$
and respective external structures:
$$\vertexQCDun,\:\vertexQCDdeux,\:\vertexQCDtrois,\:\vertexQCDquatre,\:\edgeelectronun,\:\edgeghostun,\:\edgegluon.$$
 \end{enumerate}

{\bf Remark}. We can extend the set of considered Feynman graphs by admiting other external structures, indexed by $k+l+1,\ldots,k+l+m$.
for $k+1\leq j \leq k+l$ and  $k+l+1\leq i\leq k+l+m$, let $\lambda^{(i)}_j$ be the number of copies of half-edges of the $j$-th type of edge $t_j$
in the $i$-th external structure, divided by $2$. We obtain a SDSE given by:
\begin{align*}
X_i&=\sum_{d\in \NN} B_{i,d}\left(\prod_{j=1}^k(1+X_j)^{\sum_p c'_{j,p}d_p} \prod_{j=k+1}^{k+l}(1-X_j)^{\sum_{p,q}a'_{j,p}c'_{p,q}d_j}(1+X_i)
\right)\mbox{ if }i\leq k,\\
X_i&=\sum_{d\in \NN} B_{i,d}\left(\prod_{j=1}^k(1+X_j)^{\sum_p c'_{j,p}d_p} \prod_{j=k+1}^{k+l}(1-X_j)^{\sum_{p,q}a'_{j,p}c'_{p,q}d_j}(1-X_i)
\right)\mbox{ if }k+1\leq i\leq k+l,\\
X_i&=\sum_{d\in \NN} B_{i,d}\left(\prod_{j=1}^k(1+X_j)^{\sum_p c'_{j,p}d_p} \prod_{j=k+1}^{k+l}(1-X_j)^{\sum_{p,q}a'_{j,p}c'_{p,q}d_j}
\prod_{j=k+1}^{k+l} (1-X_j)^{\lambda^{(i)}_j}
\right)\mbox{ otherwise}.
\end{align*}
We recognize the deg1 pre-Lie algebra with $I_p=\{p\}$ for all $1\leq p\leq k+l$, $I_0=\{k+l+1,\ldots,k+l+n\}$, $b$ given by:
$$b^{(i,j)}=\begin{cases}
\delta_{i,j}\mbox{ if }i\leq k,\\
-\delta_{i,j}\mbox{ if }k+1\leq i\leq k+l,\\
-\lambda^{(i)}_j \mbox{ if }i\geq k+l+1,
\end{cases}$$
and $A$ given by the matrix $\left(\begin{array}{c} C'\\-A'_\calT C'\end{array}\right)$.
If $C$ is invertible, $\h_{(S'_\calT)}$ is isomorphic to the Hopf algebra of coordinates of the group:
$$(V_{\lambda^{(k+l+1)},0}\oplus \ldots \oplus V_{\lambda^{(k+l+m)},0})\rtimes(V_0^l \rtimes \bfG_{A''}).$$

\subsection{Minimal rank for QCD}

Let us consider a QFT, the SDSE $(S'_\calT)$ associated to it, and a matrix $C$ giving a connected $\N^N$-graduation.
We proved that if $Rank(C)=|\calV|$, then $\h_{(S_\calT)}$ is Hopf; we would like to know what the minimal rank of $C$ required to make $\h_{(S'_\calT)}$
a Hopf subalgebra is. For QED or $\varphi^n$, as $|\calV|=1$, this is obviously $1$. 
If the theory has \emph{enough} primitive Feynman graphs, this minimal rank is $|\calV|$: we now prove this result for QCD.

\begin{prop}\label{propQCD}
In the QCD case, the graduation induced by $C$ gives a Hopf SDSE if, and only if $Rank(C)=4$.
\end{prop}

\begin{proof} We already proved the implication $\Longleftarrow$. 
We first construct enough primitive Feynman graphs of external structure $\vertexQCDun$.
Let $(a,b,c,d) \in \N_*^4$. We start with $G=\vertexQCDun$.
Judiciously gluing the external edges of  $a$ copies of $\edgegluon$, $b$ copies of $\QCDdeuxbis$, 
$c$ copies of $\vertexQCDtrois$ and $d$ copies of$\vertexQCDquatre$  on the edges
$\edgeelectronun$ and $\edgeelectrondeux$, creating in this way new $2a+2b+3c+4d$ new vertices of type $\vertexQCDun$, 
we obtain a primitive Feynman graph $G'$ with:
$$V_{G'}=\left(\begin{array}{c} 1\\0\\0\\0\end{array}\right)+a\left(\begin{array}{c}2\\0\\0\\0\end{array}\right)
+b\left(\begin{array}{c}2\\2\\0\\0\end{array}\right)+c\left(\begin{array}{c}3\\0\\1\\0\end{array}\right)+d\left(\begin{array}{c}4\\0\\0\\1\end{array}\right).$$

Let us assume that $\h_{(S'_\calT)}$ is Hopf and that $Rank(C)\leq 3$. There exists a nonzero vector $v \in \mathbb{Q}^4$, such that 
$Cv=(0)$. We decompose this vector $v$ in the basis:
$$\left(\left(\begin{array}{c}2\\0\\0\\0\end{array}\right), \left(\begin{array}{c}2\\2\\0\\0\end{array}\right),
\left(\begin{array}{c}3\\0\\1\\0\end{array}\right), \left(\begin{array}{c}4\\0\\0\\1\end{array}\right)\right).$$
After a multiplication by a nonzero integer and separation of the terms according to their signs, 
we obtain that there exists two different vectors $w$ and $w'$, such that:
\begin{align*}
w&=a\left(\begin{array}{c}2\\0\\0\\0\end{array}\right)+b\left(\begin{array}{c}2\\2\\0\\0\end{array}\right)
+c\left(\begin{array}{c}3\\0\\1\\0\end{array}\right)+d\left(\begin{array}{c}4\\0\\0\\1\end{array}\right),\: a,b,c,d\in \N,\\
w'&=a'\left(\begin{array}{c}2\\0\\0\\0\end{array}\right)+b'\left(\begin{array}{c}2\\2\\0\\0\end{array}\right)
+c'\left(\begin{array}{c}3\\0\\1\\0\end{array}\right)+d'\left(\begin{array}{c}4\\0\\0\\1\end{array}\right),\: a',b',c',d'\in \N,\\
Cw&=Cw'.
\end{align*}
Let $G$ and $G'$ be primitive Feynman graphs of external structure $\vertexQCDun$ such that:
\begin{align*}
V_G&=\left(\begin{array}{c} 1\\0\\0\\0\end{array}\right)+w,&V_{G'}&=\left(\begin{array}{c} 1\\0\\0\\0\end{array}\right)+w'.
\end{align*}
Their degree are:
\begin{align*}
deg(G)&=CV_G-(C\: 0)\left(\begin{array}{c}1\\0\\\vdots\\0\end{array}\right)=Cw,&
deg(G')&=CV_{G'}-(C\: 0)\left(\begin{array}{c}1\\0\\\vdots\\0\end{array}\right)=Cw'.
\end{align*}
So $deg(G)=deg(G')$. According to lemma \ref{lemmaseriesidentiques}, $f_G=f_{G'}$, so in particular, for all $i\in [4]$, $v_i(G)=v_i(G')$,
which implies that $V_G=V_{G'}$ and finally $w=w'$, which is a contradiction. We conclude that $Rank(C)=4$. \end{proof}

\section{SDSE associated to coloured graphs}

We now generalize multicylic SDSE of \cite{Foissy3,Foissy6}. We are interested here in SDSE of the form:
$$(S): \forall i\in [M], \: X_i=\sum_{\alpha \in \N^N_*} B_{i,\alpha}\left(1+\sum_{j\in I_{i,\alpha}} X_j\right),$$
where the $I_{i,\alpha}$ are nonempty sets.

\begin{defi}
\begin{enumerate}
\item A $N$-coloured oriented graph is an oriented graph $G$, with a map from the set $E(G)$ of edges of $G$ into $[N]$.
We denote by $V(G)$ the set of vertices of $G$. 
For all $i,j \in V(G)$, for all $\alpha=(\alpha_1,\ldots,\alpha_N)\in \N^N$, we shall write $i\fleche{\alpha} j$ if there exists an oriented path from $i$ to $j$
in $G$, with $\alpha_i$ edges coloured by $i$ for all $i\in [N]$.
\item Let $G$ be a $N$-coloured oriented graph. The SDSE associated to $G$ is associated to the $\N^N$-graded partitioned
set $D=V(G)\times \NN$:
$$(S_G):\:\forall i\in V(G), \: X_i=\sum_{\alpha \in \NN} B_{i,\alpha}\left(1+\sum_{i\fleche{\alpha} j} X_j\right).$$ 
\item Let $G$ be a $N$-coloured oriented graph. We shall say that $G$ is Hopf if, for all $i,j,k \in V(G)$, for all $\alpha,\beta \in \NN$,
$$(i\fleche{\alpha} j \mbox{ and }j\fleche{\beta}k)\Longleftrightarrow (i\fleche{\alpha} j  \mbox{ and }i\fleche{\alpha+\beta}k).$$
(Note that $\Longrightarrow$ is always satisfied). 
\end{enumerate}\end{defi}

\begin{prop}
We consider a SDSE of the form:
$$(S): \forall i\in [M], \: X_i=\sum_{\alpha \in \N^N_*} B_{i,\alpha}\left(1+\sum_{j\in I_{i,\alpha}} X_j\right),$$
It is Hopf, if, and only if, there exists a $N$-coloured Hopf graph on $[M]$ such that $(S)$ is equal to $(S_G)$. 
If this holds, the dual pre-Lie algebra of $\h_{(S_G)}$ is associative; 
it has a basis $(f_i(\alpha))_{i\in V(G),\alpha\in \NN}$ and the product is given by:
$$f_j(\beta)*f_i(\alpha)=\begin{cases}
f_i(\alpha+\beta)\mbox{ if }i\fleche{\alpha} j,\\
0\mbox{ otherwise}.
\end{cases}$$
\end{prop}

\begin{proof} $\Longrightarrow$. {\it First step.} Let us assume that $(S)$ is Hopf. We fix $i,j,k \in [M]$ and $\alpha,\beta,\gamma \in \N^N_*$.
We put:
\begin{align*}
a&=\begin{cases}
1&\mbox{ if }j\in I_{i,\alpha},\\
0&\mbox{ otherwise;}
\end{cases}&
b&=\begin{cases}
1&\mbox{ if }k\in I_{i,\alpha+\beta},\\
0&\mbox{ otherwise;}
\end{cases}&
c&=\begin{cases}
1&\mbox{ if }k\in I_{j,\beta},\\
0&\mbox{ otherwise.}
\end{cases}&
\end{align*}
We obtain:
\begin{align*}
X_i(\alpha+\beta)&=\tdun{$i,\alpha+\beta$}\hspace{.7cm}+a\tddeux{$i,\alpha$}{$j,\beta$}\hspace{.25cm}+\ldots,\\
X_i(\alpha+\beta+\gamma)&=\tdun{$i,\alpha+\beta+\gamma$}\hspace{10.5mm}+b\tddeux{$i,\alpha+\beta$}{$k,\gamma$}\hspace{.7cm}
+ac\tdtroisdeux{$i,\alpha$}{$j,\beta$}{$k,\gamma$}\hspace{.25cm}+\ldots
\end{align*}
Hence:
$$\Delta(X_i(\alpha+\beta+\gamma))=\tdun{$k,\gamma$}\hspace{.25cm}\otimes
\underbrace{\left(b\tdun{$i,\alpha+\beta$}\hspace{.7cm}+ac\tddeux{$i,\alpha$}{$j,\beta$}\hspace{.25cm}+\ldots\right)}_{=X}\ldots$$
As $\h_{(S)}$ is Hopf, $X$ is a multiple of $X_i(\alpha+\beta)$, so $ab=ac$: $a=0$ or $b=c$. In particular, if $a\neq 0$, $b=c$. Hence, 
for all $i,j,k \in [M]$, for all $\alpha,\beta \in \N^N_*$, 
$j\in I_{i,\alpha}$ and $k\in I_{j,\beta}$ if, and only if, $i\in I_{i,\alpha}$ and $k\in I_{i,\alpha+\gamma}$.\\

{\it Second step.} We define a coloured graph structure $G$ on $[M]$ in the following way: for all $i,j \in [M]$, for all $p\in [N]$, 
there exists an edge from $i$ to $j$ decorated by $p$ if, and only if, $j\in I_{i,\epsilon_p}$. Let us prove that for all $i,k \in [M]$,
for all $\alpha \in \N^N_*$, $k\in I_{i,\alpha}$ if, and only if, $i\fleche{\alpha}k$ in $G$. We proceed by induction on $|\alpha|=\alpha_1+\ldots+\alpha_N$.
This is obvious if $|\alpha|=1$. Let us assume the result at rank $|\alpha|-1$.

$\Longrightarrow$. Let us choose $\alpha'$ and $\alpha''$, such that $\alpha=\alpha'+\alpha''$, $|\alpha'|=|\alpha|-1$ and $|\alpha''|=1$.
Let $j\in I_{i,\alpha'}$. By the first step, then $k\in I_{j,\alpha''}$. By the induction hypothesis, $i\fleche{\alpha'} j$ and $j\fleche{\alpha''} k$,
so $i\fleche{\alpha} k$.

$\Longleftarrow$. Let us assume that $i\fleche{\alpha} k$. We consider a path in $G$ from $i$ to $k$, of weight $\alpha$.
If $j$ is the last step, there exists $\alpha',\alpha''$, such that $\alpha=\alpha'+\alpha''$, $|\alpha'|=|\alpha|-1$ and $|\alpha''|=1$,
$i\fleche{\alpha'} j$ and $j\fleche{\alpha''} k$. By the induction hypothesis, $j\in I_{i,\alpha'}$ and $k\in I_{j,\alpha''}$.
By the first step, $k\in I_{i,\alpha}$.

Hence:
$$(S):\:\forall i\in [M], \: X_i=\sum_{\alpha \in \NN} B_{i,\alpha}\left(1+\sum_{i\fleche{\alpha} j} X_j\right).$$ 
So $(S)=(S_G)$. The first step implies that $(G)$ is Hopf.\\


$\Longrightarrow$. Let us consider a Hopf $N$-coloured graph $G$. We define a product on the vector space 
$\g=Vect(f_i(\alpha))_{i\in V(G), \alpha \in \NN}$ by:
$$f_j(\beta)*f_i(\alpha)=\begin{cases}
f_i(\alpha+\beta)\mbox{ if }i\fleche{\alpha} j,\\
0\mbox{ otherwise}.
\end{cases}$$
Let $i,j,k\in V(G)$. For any $\alpha,\beta,\gamma \in \NN$:
\begin{align*}
(f_k(\gamma)*f_j(\beta))*f_i(\alpha)&=\begin{cases}
f_i(\alpha+\beta+\gamma)\mbox{ if }i\fleche{\alpha} j \mbox{ and }j\fleche{\beta}k,\\
0\mbox{ otherwise};
\end{cases}\\
f_k(\gamma)*(f_j(\beta)\*f_i(\alpha))&=\begin{cases}
f_i(\alpha+\beta+\gamma)\mbox{ if }i\fleche{\alpha} j \mbox{ and }i\fleche{\alpha+\beta}k,\\
0\mbox{ otherwise}.
\end{cases}
\end{align*}
So $*$ is associative. Hence, $\g$ is a $\N^N$-graded connected pre-Lie algebra. It is not difficult to prove that the graded dual of its enveloping
algebra, imbedded in $\h^D$, is the subalgebra generated by the solution of the SDSE $(S_G)$, which as a consequence is Hopf. \end{proof}\\

An example of coloured graph is given by families of commuting endofunctions:

\begin{prop}
Let $V$ be a set and, for all $1\leq p\leq [N]$, let $f_p:V\longrightarrow V$ be a map. We construct a graph $G_f$ in the following way:
\begin{itemize}
\item $V(G)=V$.
\item For all $i,j \in V$, for all $p\in [N]$, $i\fleche{\epsilon_p} j$ if, and only if, $f_p(i)=j$.
\end{itemize}
In other terms, $G$ is the coloured graph of maps $f_1,\ldots,f_N$. Then $G_f$ is Hopf, if and only if, for all $p,q\in [N]$,
$f_p \circ f_q=f_q\circ f_p$.
\end{prop}

\begin{proof} $\Longrightarrow$. Let us assume that $G_f$ is Hopf. Let $i,j \in V$. We put $j=f_q(i)$, $j'=f_p(i)$,
$k=f_q\circ f_p(i)$, $\alpha=\epsilon_q$ and $\beta=\epsilon_p$. Then $i\fleche{\alpha} j$ and $i\fleche{\alpha+\beta} k$,
so $j\fleche{\beta} k$: hence, $f_p(j)=f_p\circ f_q(j)=k=f_q\circ f_p(i)$.\\

$\Longleftarrow$. Let us assume that $i\fleche{\alpha} j$ and $i\fleche{\alpha+\beta}k$.
There exists a sequence $p_1,\ldots,p_m$ such that $\epsilon_{p_1}+\ldots+\epsilon_{p_m}=\alpha$, 
$f_{p_1}\circ \ldots \circ f_{p_m}(i)=j$. There exists a sequence 
$p_1,\ldots,p_m$ such that $\epsilon_{q_1}+\ldots+\epsilon_{q_{m+n}}=\alpha+\beta$, 
$f_{q_1}\circ \ldots \circ f_{q_{m+n}}(i)=k$. Note that $m=|\alpha|$, and $m+n=|\alpha+\beta|$.
Moreover, the multiset $\{p_1,\ldots,p_m\}$ is included in the multiset $\{q_1,\ldots,q_{m+n}\}$.
We proceed by induction on $n=|\beta|$. If $n=1$, let us assume that $\{q_1,\ldots,q_{m+n}\}\setminus \{p_1,\ldots,p_m\}=\{q_r\}$.
The commutation relation implies that:
$$f_{q_1}\circ \ldots \circ f_{q_{m+n}}(i)=f_r \circ f_{q_1}\circ \ldots f_{q_{r-1}}\circ f_{q_{r+1}}\circ \ldots \circ f_{q_{m+1}}(i).$$
By permuting the $p_j$'s using the commutation relations, we obtain that these two elements are equal. Hence, $f_r(j)=l$,
so $j\fleche{\beta} k$.

Let us assume the result at rank $n-1$. There exists $k'$, such that $i\fleche{\alpha+\beta'} k'$ $k'\fleche{\beta''} k$,
$\beta=\beta'+\beta''$, $|\beta'|=n-1$ and $|\beta''|=1$. By the induction hypothesis, $j\fleche{\beta''} k'$, so $j\fleche{\beta} k$. \end{proof}\\





{\bf Example.}  Let $V=\Z/N\Z$, $N=1$ and:
$$f:\left\{\begin{array}{rcl}
\Z/N\Z&\longrightarrow&\Z/N\Z\\
\overline{k}&\longrightarrow&\overline{k+1}.
\end{array}\right.$$
The SDSE associated to $f$ is a cyclic SDSE of \cite{Foissy3,Foissy6}.\\

{\bf Remark.} There are other examples of coloured Hopf graphs, for example:
$$\xymatrix{\bullet\ar[d]^1\ar[rd]^1&\bullet\ar[d]^2\\ \bullet&\bullet}$$

\bibliographystyle{amsplain}
\bibliography{biblio}
\end{document}